\documentclass{siamart0216}
\usepackage{amsmath}
\usepackage{amssymb}
\usepackage{bbold}
\usepackage{graphicx}
\usepackage{enumerate}
\usepackage{color}
\usepackage{tikz}
\usepackage{subfig}
\usepackage{algorithmic}
\Crefname{ALC@unique}{Line}{Lines}
\usetikzlibrary{arrows,intersections}
\usepackage{pgfplots}
\usepackage{braket,amsfonts,amsopn}

\newsiamremark{remark}{Remark}
\newsiamremark{example}{Example}

\newcommand{\bB}{\mathbf{B}}
\newcommand{\bF}{\mathbf{F}}
\newcommand{\bN}{\mathbf{N}}
\newcommand{\bR}{\mathbf{R}}
\newcommand{\bS}{\mathbf{S}}
\newcommand{\bb}{\mathbf{b}}
\newcommand{\bsig}{\mathbf{\sigma}}
\newcommand{\cN}{\mathcal{N}}
\newcommand{\cT}{\mathcal{T}}
\newcommand{\rF}{\mathsf{F}}
\newcommand{\rL}{\mathsf{L}}
\newcommand{\Ome}{{\Omega}}
\newcommand{\p}{\partial}
\newcommand{\nab}{\nabla}
\newcommand{\oOme}{\overline{\Ome}}
\newcommand{\hmax}{\mathbb{h}}
\newcommand{\USC}{\mathrm{USC}}
\newcommand{\LSC}{\mathrm{LSC}}
\newcommand{\Id}{\mathrm{Id}}

\DeclareMathOperator{\trace}{tr}
\DeclareMathOperator{\diam}{diam}

\newcommand{\bi}[1]{\mbox{\boldmath$#1$\unboldmath}}

\newcommand{\bsigma}{\bi\sigma}
\newcommand{\blambda}{\bi\lambda}

\definecolor{dred}{RGB}{200,80,80}
\definecolor{dgreen}{RGB}{70,140,70}
\definecolor{dblue}{RGB}{80,80,200}

\newcommand{\blue}[1]{#1}
\newcommand{\red}[1]{#1}

\begin{document}

\title{Convergent semi-Lagrangian methods for the Monge-Amp\`ere equation on unstructured grids} 

\author{Xiaobing Feng\thanks{Department of Mathematics, The University of Tennessee, Knoxville, TN 37996, U.S.A. (xfeng@math.utk.edu). The work of this author was partially supported by the NSF grant DMS-0710831.}
\and
Max Jensen\thanks{Department of Mathematics, University of Sussex, Brighton BN1 9QH, United Kingdom (m.jensen@sussex.ac.uk). }
}

\maketitle
\date{\today}
\begin{abstract}
This paper is concerned with developing and analyzing convergent semi-Lagrangian methods for the fully nonlinear elliptic Monge-Amp\`ere equation on general triangular grids. This is done by establishing an equivalent (in the viscosity sense) Hamilton-Jacobi-Bellman formulation of the Monge-Amp\`ere equation. A significant benefit of the reformulation is the removal of the convexity constraint from the admissible space as convexity becomes a built-in property of the new formulation. Moreover, this new approach allows one to tap the wealthy numerical methods, such as semi-Lagrangian schemes, for Hamilton-Jacobi-Bellman equations to solve Monge-Amp\`ere type equations. It is proved that the considered numerical methods are monotone, pointwise consistent and uniformly stable. Consequently, its solutions converge uniformly to the unique convex viscosity solution of the Monge-Amp\`ere Dirichlet problem. A superlinearly convergent Howard's algorithm, which is \blue{a Newton--type method}, is utilized as the nonlinear solver to take advantage of the monotonicity of the scheme. Numerical experiments are also presented to gauge the performance of the proposed numerical method and the nonlinear solver.
\end{abstract}

\begin{keywords}
Monge-Amp\`ere equation, Hamilton-Jacobi-Bellman equation, viscosity solution, semi-Lagrangian method, wide stencil, monotone scheme, convergence, Howard's algorithm.  
\end{keywords}

\begin{AMS}
65N06, 
65N12, 
65N35, 
35J60
\end{AMS}

\pagestyle{myheadings}
\thispagestyle{plain}
\markboth{XIAOBING FENG AND MAX JENSEN}{SEMI-LAGRANGIAN METHODS FOR MONGE-AMP\`ERE EQUATION}

\section{Introduction}\label{sec:intro}

This paper is concerned with semi-Lagrangian methods for the following Dirichlet boundary value problem of a fully nonlinear elliptic Monge-Amp\`ere-type equation:
\begin{subequations} \label{MA_intro}
\begin{alignat}{2}\label{MA}
\det(D^2 u) &= \Bigl( \frac{f}{d} \Bigr)^d &&\qquad \mbox{in } \Ome, \\
u(x) &= g(x) &&\qquad \mbox{on } \p\Ome, \label{BC}
\end{alignat}
\end{subequations}
where $\Ome$ and $\p \Ome$ denote respectively a bounded strictly convex domain in $\bR^d \,(d\geq 2)$ and its boundary. The Hessian of the function $u$ is denoted $D^2u$. The functions $f: \Ome \to [0,\infty)$ and $g : \partial \Ome \to \bR$ are bounded and continuous. We note that the special form of the right-hand side in \cref{MA} is chosen for the notational convenience in the subsequent analysis; the usual form can be easily recovered by setting $f=d\widetilde{f}^{\frac{1}{d}}$.

Monge-Amp\`ere type equations, along with Hamilton-Jacobi-Bellman type equations (see below), are two major classes of fully nonlinear second order partial differential equations (PDEs). They arise from many scientific and technological applications such as antenna design, astrophysics, differential geometry, image processing, optimal mass transport and semi-geostrophic fluids, \blue{to name a few} (see \cite[Section 5]{FGN13} for details). From the PDE point of view, Monge-Amp\`ere type equations are well understood, see \cite[Chapter 17]{Gilbarg_Trudinger01} for a detailed account on the classical solution theory and \cite{Gutierrez01, Caffarelli_Cabre95} for the viscosity solution theory. On the other hand, from the numerical point of view, the situation is far from ideal. Very few numerical methods, which can reliably and efficiently approximate viscosity solutions of Monge-Amp\`ere type PDEs on general convex domains, are available in the literature (see \cite{BGNS09, FGN13, FKL14, Feng_Neilan09a, Lakkis_Pryer13, Oberman08} and the references therein).  
There are three main difficulties which lead to \blue{the lack of progress on approximating viscosity solutions of fully nonlinear second order PDEs}. Firstly, the fully nonlinear structure and nonvariational concept of viscosity solutions of the PDEs prevent a direct formulation of any Galerkin-type numerical methods (such as finite element, discontinuous Galerkin and spectral methods). Secondly, the Monge-Amp\`ere operator, $u \mapsto \det (D^2u)$, is not an elliptic operator in generality, instead, it is only elliptic in the set of convex functions and the uniqueness of viscosity solutions only holds in that space. This convexity constraint, imposed on the admissible space, causes a daunting challenge for constructing convergent numerical methods; it indeed screens out any trivial finite difference and finite element analysis because the set of convex finite element functions is not dense in the set of convex functions~\cite{Aguilera_Morin09}. Thirdly, as the right-hand side $f$ of \cref{MA} vanishes, the Monge-Amp\`ere mapping attains characteristics of a degenerate elliptic operator. In this setting the regularity of exact solutions is reduced, limiting the tools available for a convergence analysis of numerical solutions.

The goal of this paper is to develop a new approach for constructing convergent numerical methods for the Monge-Amp\`ere Dirichlet problem \cref{MA_intro}, in particular, by focusing on overcoming the second difficulty caused by the convexity constraint. The crux of the approach is to first establish an equivalent (in the viscosity sense) Bellman formulation of the Monge-Amp\`ere equation and then to design monotone semi-Lagrangian methods for the resulting Bellman equation on general triangular grids. The proposed methods are closely related to two-grid constructions because we use a finite element ambient grid to define the approximation space, combined with wide finite-difference stencils layered over this ambient grid. An aim in the design of the numerical schemes is to make Howard's algorithm available, which is a globally superlinearly converging semi-smooth Newton solver. This allows us to robustly compute numerical approximations on very fine meshes of non-smooth viscosity solutions, \blue{including the degenerate case where $f\geq 0$.} An advantage of the rigorous convergence analysis of the numerical solutions is the comparison principle for the Bellman operator, which extends to non-convex functions. We deviate from the established Barles-Souganidis framework in the treatment of the boundary conditions to address challenges arising from consistency and comparison. The proposed approach also bridges the gap between advances on numerical methods for these two classes of second order fully nonlinear PDEs, see for instance \cite{Bonnans_Zidani03, Camilli_Falcone95, Camilli_Jakobsen09, Debrabant_Jakobsen12, Jensen_Smears12, Kushner90, Smears_Suli14} and the references therein for the numerical literature on Bellman equations.

The remainder of this paper is organized as follows. In \cref{sec:viscsol} we collect preliminaries including the definition of viscosity solutions. In \cref{sec:MA_HJB} we introduce a well-known Hamilton-Jacobi-Bellman reformulation of the Monge-Amp\`ere equation in the classical solution setting and prove such an equivalence still holds in the viscosity solution framework. In \cref{sec:SLmethods} we introduce a numerical scheme \cref{eq:disc} for the Monge-Amp\`ere equation. In \cref{sec-dwp} prove the existence and uniqueness of numerical solutions and present a globally converging semi-smooth Newton method. \Cref{sec:convergence} contains the main result of the paper: \Cref{thm:main} demonstrates the uniform convergence to the unique viscosity solution. In \cref{sec-ps} we relate the class of schemes of this paper to existing methods to solve Hamilton-Jacobi-Bellman equations. In \cref{sec:experiments} we present numerical experiments which verify the accuracy and efficiency of the proposed method and the nonlinear solver.

\section{Viscosity solutions} \label{sec:viscsol}

Let $\Ome\subset\mathbf{R}^d$ be a bounded open strictly convex domain. We denote by $B(G)$, $\USC(G)$, and $\LSC(G)$, respectively, the spaces of bounded, upper semi-continuous, and lower semicontinuous functions on a set $G \subset \mathbf{R}^d$. For any $v\in B(\oOme)$, we define
\[
v^*(x):=\limsup_{y\to x} v(y) \qquad\mbox{and}\qquad
v_*(x):=\liminf_{y\to x} v(y). 
\]
Then, $v^*\in \USC(\oOme)$ and $v_*\in \LSC(\oOme)$, and they are called the {\em upper} and {\em lower semicontinuous envelopes} of $v$, respectively.

Given a bounded function $F: \bS\times\mathbf{R}^d\times \mathbf{R}\times \Ome \to \mathbf{R}$, where $\bS$ denotes the set of $d\times d$ symmetric real matrices, the general second-order fully nonlinear PDE takes the form
\begin{align}\label{eqn:F}
F(D^2u,\nabla u, u, x) = 0 \qquad\mbox{in } \Ome.
\end{align}
We impose Dirichlet boundary conditions in the pointwise sense that $u(x) = g(x)$ for all $x \in \p\Ome$. In the discussion about converging numerical schemes we shall draw comparisons with Dirichlet conditions in the viscosity sense, which are imposed as a discontinuity of the PDE, cf.~\cite[p.274]{Barles_Souganidis91} and \cite[Section 7.C]{Crandall_Ishii_Lions92}. 

The following definitions can be found in \cite{Barles_Souganidis91,Caffarelli_Cabre95,Crandall_Ishii_Lions92,Gilbarg_Trudinger01,Gutierrez01}.

\begin{definition}\label{def:viscsol}
A function $u\in \USC(\blue{\Ome})$ (resp.~$u\in \LSC(\blue{\Ome})$) is called a viscosity subsolution (resp.~supersolution) of \cref{eqn:F} if for all $\varphi \in C^2(\Ome)$ such that $u-\varphi$ has a local maximum (resp.~minimum) at $x\in \Ome$ we have
\[
F(D^2\varphi(x),\nabla \varphi(x), u(x), x) \leq 0 
\]
(resp.~$F(D^2\varphi(x),\nabla \varphi(x), u(x), x) \geq 0$). The function $u$ is said to be a viscosity solution of \cref{eqn:F} if it is simultaneously a viscosity subsolution and supersolution of \cref{eqn:F}.
\end{definition}

The restriction to convex functions in \cref{def:cviscsol} below reflects that the Monge-Amp\`ere equation is only elliptic on the set of convex functions, while the Hamilton-Jacobi-Bellman operator of our subsequent construction is elliptic on the whole space. For details we refer to \cite[Section 1.3]{Gutierrez01}.

\begin{definition}\label{def:cviscsol}
A function $u\in \USC(\blue{\Ome})$ (resp.~$u\in \LSC(\blue{\Ome})$) is called a viscosity subsolution (resp.~supersolution) of \cref{eqn:F} on the set of convex functions if $u$ is convex and if for all convex $\varphi \in C^2(\Ome)$ such that \blue{$u-\varphi$} has a local maximum (resp.~minimum) at $x\in \Ome$ we have
\[
F(D^2\varphi(x),\nabla \varphi(x), u(x), x) \leq 0 
\]
(resp.~$F(D^2\varphi(x),\nabla \varphi(x), u(x), x) \geq 0$). The function $u$ is said to be a viscosity solution of \cref{eqn:F} on the set of convex functions if it is simultaneously a viscosity subsolution and supersolution of \cref{eqn:F} on the set of convex functions.
\end{definition}

Note that in \cref{def:cviscsol} the set of test functions is smaller. Therefore it is not obvious that viscosity solutions on the set of convex functions are solutions in the sense of \cref{def:viscsol}.

\section{Hamilton-Jacobi-Bellman form of the Monge-Amp\`ere equation} \label{sec:MA_HJB}

It is known \cite{Krylov87,Lions85} that the Monge-Amp\`ere equation has an equivalent Hamilton-Jacobi-Bellman (or Bellman for brevity) formulation in the setting of classical solutions. However, to the best of our knowledge, such an equivalence has not been extended to the case of viscosity solutions in the literature. The goal of this section is to prove this extension rigorously. \blue{A related description of the relationship between classical and viscosity solutions is examined in terms of elliptic sets in \cite{Krylov97}.}

Let $\bS_+ := \{ A \in \bS;\, A \ge 0 \}$ and $\bS_1 := \{ B \in \bS_+;\, \trace B = 1 \}$. It is easy to check \cite{Krylov87} that $\bS_1$ is a compact subset of $\bS_+$ and, consequently, $\bS_1$ is bounded in the Euclidean norm.

We define the Bellman operator
\begin{alignat}{2}\label{HJB_operator}
H(A,f) &:=\sup_{B \in \bS_1} \Bigl( -B:A + f \sqrt[d]{\det B} \Bigr) 
&&\qquad \forall A\in \bS, f \in [0,\infty),
\end{alignat}
and the Monge-Amp\`ere operator
\begin{align}\label{MA_operator}
M(A,f) := \Bigl( \frac{f}{d} \Bigr)^d - \det(A) \qquad\forall A\in \bS, f \in [0,\infty).
\end{align}
Then the Monge-Amp\`ere problem \cref{MA_intro} can be rewritten as
\begin{subequations} \label{EMA}
\begin{align} \label{EMAeq}
M\bigl(D^2u(x),f(x)\bigr) & = 0 && \forall x\in \Ome,\\[1mm]
u(x) & = g(x) && \forall x\in \p\Ome, \label{EMAbc}
\end{align}
\end{subequations}
which gives the structure of \cref{eqn:F} upon setting
\[
F\bigl(D^2u(x),\nab u(x),u(x),x\bigr) = M\bigl(D^2u(x),f(x)\bigr).
\]
Analogously we also define the Bellman problem
\begin{subequations} \label{EHJB}
\begin{align} \label{EHJBeq}
H\bigl(D^2u(x),f(x)\bigr) & = 0 && \forall x\in \Ome,\\[1mm]
u(x) & = g(x) && \forall x\in \p\Ome, \label{EHJBbc}
\end{align}
\end{subequations}
with the correspondence $F\bigl(D^2u(x),\nab u(x),u(x),x\bigr)= H\bigl(D^2u(x),f(x)\bigr)$.

The proofs of the following \cref{sdiag} and \cref{cequi} are given in \cite[p.51]{Krylov87}.

\begin{lemma} \label{sdiag}
There exists a maximizer $B' \in \bS_1$ of the supremum in \cref{HJB_operator} which commutes with $A \in \bS$. In particular, there is a coordinate transformation, depending on $A$, which simultaneously diagonalizes $A$ and $B'$. 
\end{lemma}

The next result gives equivalence of convex classical solutions of \cref{EMA} and \cref{EHJB}. We highlight that the lemma covers the degenerate case $f = 0$.

\begin{lemma} \label{cequi}
Let $f \in [0,\infty)$ and $A \in \bS$. Then $H(A,f) = 0$ holds if and only if $M(A,f) = 0$ and $A \in \bS_+$.
\end{lemma}

We remark that there is another slightly different Bellman reformulation of the Monge-Amp\`ere problem \cref{EMA} which uses a determinant constraint (instead of a trace constraint) on the control $B$ in the definition of the Hamiltonian $H$, see \cite{Lions85}. However, the numerical discretization of a determinant constraint is less straightforward, explaining our preference for \cref{HJB_operator}.

Let $D_\ell$ be the matrix $(\delta_{i\ell} \delta_{j \ell})_{ij}$ which vanishes in all entries except for the $\ell$th diagonal term which is $1$.

\begin{theorem}\label{thm:MA_to_HJB}
Let $f \in C(\Ome)$ be non-negative and $u$ be a viscosity subsolution (supersolution) of the Monge-Amp\`ere problem \cref{EMAeq} on the set of convex functions. Then $u$ is a viscosity subsolution (supersolution) of Bellman problem \cref{EHJBeq}.
\end{theorem}

\begin{proof}
{\em Step 1}: We first consider the case that $u$ is a viscosity subsolution of \cref{EMAeq}. Let $\phi\in C^2(\Ome)$ such that $u-\phi$ attains a local maximum at $x\in \Ome$. Since $u$ is convex it follows that $\phi$ is convex in a neighborhood $N$ of $x$, cf.~\cite[Remark 1.3.2]{Gutierrez01}. By the definition of viscosity subsolutions on the set of convex functions, noting the local character of the definition, we have $M\bigl(D^2\phi(x),f(x)\bigr)\leq 0$. 

Let $\xi\geq 0$ such that $M(D^2\phi(x),f(x))+\xi=0.$ Equivalently, 
\[
M(D^2\phi(x),\widehat{f})=0 \quad\mbox{with}\quad \widehat{f}:= d\sqrt[d]{\Bigl(\frac{f(x)}{d}\Bigr)^d +\xi} \ge f(x).
\]
By \cref{cequi} we have $H(D^2\phi(x),\widehat{f}) = 0$. Thus, $u$ is a viscosity subsolution of~\cref{EHJBeq}, using that $g \mapsto H(D^2\phi(x),g)$ is monotonically increasing.

\medskip
{\em Step 2:} Now we consider the case that $u$ is a viscosity supersolution of \cref{EMAeq}. The proof of this step differs because now non-convex $\phi$ which are test functions for $H$ but not $M$ need to be considered and because a negative slack variable $\xi$ can in general not be covered by \cref{cequi}.

Let $\phi\in C^2(\Ome)$ such that $u-\phi$ attains a local minimum at $x\in \Ome$. 

{\em (a)} We first suppose that $\phi$ is convex in a neighborhood of $x$. Then we have $M\bigl(D^2\phi(x),f(x)\bigr)\geq 0$ and that
\[
\Bigl( \frac{f}{d} \Bigr)^d \ge \det(D^2\phi(x)) \ge 0.
\]
Hence with $\widehat{f} := d \sqrt[d]{\det(D^2\phi(x))}$ there holds $f(x) \ge \widehat{f} \ge 0$ and $M(D^2\phi(x),\widehat{f}) = 0$. Due to \cref{cequi}, $H(D^2\phi(x),f(x)) \ge H(D^2\phi(x),\widehat{f}) = 0$.

{\em (b)} Now suppose that $\phi$ is not convex in the vicinity of $x$. We may assume without loss of generality that $D^2\phi(x)$ is diagonal. Then there is a $\partial^2_{\ell\ell} \phi(x) \le 0$. Therefore
\begin{align*}
H(D^2\phi(x),f(x))
& \ge - D_\ell : D^2\phi(x) = - \partial^2_{\ell\ell} \phi(x) \ge 0.
\end{align*}

Parts (a) and (b) guarantee that $u$ is a viscosity supersolution of
\cref{EHJBeq}.
\end{proof}

To show that solutions of the Bellman problem solve the Monge-Amp\`ere problem, convexity needs to be enforced. We first prove a technical lemma.

\begin{lemma} \label{Hinv}
Let $A \in \bS_+$, $f \in [0,\infty)$ and let $\lambda$ be the smallest eigenvalue of $A$. Then the function
\[
\Phi_{A,f} : \; [-f, \infty) \to [-\lambda, \infty), \; \delta \mapsto H(A, f + \delta)
\]
is continuous, strictly monotonically increasing and bijective.
\end{lemma}

\begin{proof}
We assume without loss of generality that $A$ is a diagonal matrix and that $\lambda$ is the first entry on the diagonal of $A$.

If $\delta = -f$ then \blue{the function value of $H(A, f + \delta)$ cannot be affected by the term $(f + \delta) \sqrt[d]{\det B}$ in \cref{HJB_operator}} for any $B \in \bS_1$. Hence $D_1 \in \bS_1$ is a maximizer in \cref{HJB_operator} and $H(A,f + \delta) = - \lambda$.

Now let $\delta > -f$ and consider $B_\alpha = \alpha \, {\rm Id} + (1 - d \alpha) D_1$. Then, as $\alpha \to 0$,
\[
- B_\alpha : A = - \alpha \, {\rm tr}A - (1 - d \alpha) \lambda = - \lambda + \mathcal{O}(\alpha).
\]
Similarly,
\[
\sqrt[d]{\det B_\alpha} (f + \delta) = \left( (1 - (d-1) \alpha) \alpha^{d-1} \right)^{\frac{1}{d}} (f + \delta) = \mathcal{O}(\alpha^{1-1/d}).
\]
It follows that there is an $\alpha \in (0,1]$ such that
\[
-B_\alpha: A + (f + \delta)  \sqrt[d]{\det B_\alpha} > -D_1 : A + (f + \delta)  \sqrt[d]{\det D_1} = - \lambda.
\]
As $D_1$ is maximizer over the set of singular matrices in $\bS_1$, it is clear that the maximizer $B'$ over all of $\bS_1$ is invertible. Let $h > 0$. Then, 
\[
\Phi_{A,f}(\delta) < -B' : A + \sqrt[d]{\det B'} (f + \delta + h) \le H(A, f + \delta + h) = \Phi_{A,f}(\delta + h).
\]
Hence $\Phi_{A,f}$ is strictly monotone and thus injective.

As supremum of affine functions, $\Phi_{A,f}$ is convex and therefore continuous. This with $\Phi_{A,f}(\delta) \ge \frac{1}{d} ( f + \delta  - {\rm tr}A )$, owing to the control $\frac{1}{d} {\rm Id} \in \bS_1$, ensures that $\Phi_{A,f}$ is surjective. 
\end{proof}

With \cref{Hinv} we can find for each $A$ a suitable $\widehat{f}$ such that $H(A,\widehat{f}) = 0$.

\begin{theorem} \label{thm:HJB_to_MA}
Let $f \in C(\Ome)$ be non-negative and $u$ be a viscosity solution of the Bellman problem \cref{EHJBeq}. Then $u$ is a viscosity solution of Monge-Amp\`ere problem \cref{EMAeq} on the set of convex functions.
\end{theorem}

\begin{proof}
{\em Step 0:} \blue{Let $x \in \Ome$ and let $(p,A)$ belong to the second-order superjet
\begin{align} \label{eq:sosj}
J^{2,+} u(x):=\left\{ (D \phi(x), D^2 \phi(x)) : \phi \in C^2 \text{ and } u - \phi \text{ has local maximum at } x \right\}.
\end{align}
Then}
\[
\sup_{B \in \bS_1} \Bigl( -B:A + f \sqrt[d]{\det B} \Bigr) \le 0
\]
due to the definition of viscosity subsolutions in terms of second-order jets instead of test functions. Thus $B : A \ge f \sqrt[d]{\det B} \ge 0$ for all $B \in \bS_1$, implying that $A \ge 0$. It follows from \cite[Lemma 1]{ALL97} that $u$ is convex on $\Ome$.

{\em Step 1:} We now show that $u$ is a viscosity subsolution of \cref{EMAeq}. Let $\phi\in C^2(\blue{\Ome})$ be convex such that $u-\phi$ attains a local maximum at $x\in \Ome$. Then $H\bigl(D^2\phi(x),f(x)\bigr)\leq 0$. Let
\[
\widehat{f} = f(x) + \Phi_{D^2\phi(x),f(x)}^{-1}(0),
\]
so that $H(D^2\phi(x),\widehat{f})=0$. Since $H(D^2\phi(x),f(x)) \le 0$ it follows from monotonicity that $\widehat{f} \ge f(x) \ge 0$. By \cref{cequi} we have $M(D^2\phi(x),\widehat{f}) = 0$. Thus, $u$ is a viscosity subsolution of \cref{EMAeq}.

{\em Step 2:} Now we show that $u$ is a viscosity supersolution of \cref{EMAeq}. Let $\phi\in C^2(\blue{\Ome})$ be convex such that $u-\phi$ attains a local minimum at $x\in \Ome$. Then we have $H\bigl(D^2\phi(x),f(x)\bigr)\geq 0$. Since $D^2\phi(x)$ is positive semi-definite we know that $H(D^2\phi(x), 0) \le 0$. So $0$ is in the domain of $\Phi_{D^2\phi(x),f(x)}^{-1}$. Set
\[
\widehat{f} = f(x) + \Phi_{D^2\phi(x),f(x)}^{-1}(0).
\]
It follows $f(x) \ge \widehat{f} \ge 0$. By \cref{cequi} we have $M(D^2\phi(x),\widehat{f}) = 0$. Thus, $u$ is a viscosity supersolution of \cref{EMAeq}.
\end{proof}

At this point we have shown that the set of viscosity solutions of the Bellman and Monge-Amp\`ere operators coincide without imposing any boundary conditions. It is clear that the solution sets also coincide if Dirichlet conditions are enforced pointwise:
\begin{align*}
& \{ \blue{v \in C(\oOme) :} \text{ viscosity solution of \cref{EMAeq}} \} \cap \{ v \in C(\oOme) : v|_{\p \Ome} = 0 \}\\
= \, &\{ \blue{v \in C(\oOme) :} \text{ viscosity solution of \cref{EHJBeq}} \} \cap \{ v \in C(\oOme) : v|_{\p \Ome} = 0 \}.
\end{align*}
We now turn to a comparison principle for the Bellman problem, which holds on the whole function space. This is an advantage over comparison principles for Monge-Amp\`ere problem, which are usually formulated for the set of convex functions.

\begin{lemma} \label{thm:comparison_principle}
Let $\blue{u \in \USC(\oOme)}$ be a subsolution and $\blue{v \in \LSC(\oOme)}$ be a supersolution of the Bellman problem \cref{EHJBeq}.  Then $u \le v$ on $\oOme$ if $u \le v$ on $\p\Ome$.
\end{lemma}

\begin{proof}
We briefly outline how the comparison argument of Section 5.C in \cite{Crandall_Ishii_Lions92} applies in this context. Suppose that $u \le v$ on $\p\Ome$ but $u(x') > v(x')$ for some $x' \in \Ome$.  For $\epsilon > 0$ set $u_\epsilon(x) := u(x) + \frac{\epsilon}{2} | x - x' |^2 - \blue{\frac{\epsilon}{2} \sup_{y\in \Omega} |y-x'|^2}$, where $| \cdot |$ denotes the Euclidean norm. \blue{Notice that} $u_\epsilon \le v$ on $\p\Ome$. Moreover, for $x \in \Ome$, one has \blue{\cite[Remark 2.7(ii)]{Crandall_Ishii_Lions92}}
\[
(p, X) \in \overline{J}^{2,+} u_\epsilon(x) \qquad \text{if and only if} \qquad (p - \nabla_x {\textstyle\frac{\epsilon}{2} | x - x' |^2}, X - \epsilon {\rm Id}) \in \overline{J}^{2,+} u(x),
\]
\blue{where we referred to the closures
\begin{align*}
\overline{J}^{2,+} u(x) := \bigl\{ & (p,X) \in \bR^d \times \bS : \exists \, (x_n, p_n, X_n) \in \Omega \times \bR \times \bS \text{ so that } \\
& (p_n, X_n) \in J^{2,+} u(x_n) \text{ and } (x_n, u(x_n), p_n, X_n) \to (x, u(x), p, X) \bigr\}
\end{align*}
of the superjets \cref{eq:sosj} as required by Theorem 3.2 of \cite{Crandall_Ishii_Lions92} used below. 

Now, with} the maximizer $B'$,
\begin{align*}
H(X,f(x)) & = \sup_{B \in \bS_1} \left( - B : X + f(x) \sqrt[d]{\det B} \right) = - B' : X + f(x) \sqrt[d]{\det B'}\\[1mm]
& = - B' : (X - \epsilon {\rm Id}) + f(x) \sqrt[d]{\det B'} - \epsilon \le H(X - \epsilon {\rm Id},f(x)) - \epsilon \le - \epsilon,
\end{align*}
where we used that $B' : \Id = \trace B' = 1$. 

We assume $\epsilon \in (0, \blue{2 (u(x') - v(x')) / \diam(\Omega)^2})$ because then $u_\epsilon(x') > v(x')$. Arguing with Proposition 3.7 of \cite{Crandall_Ishii_Lions92}, for $\alpha$ sufficiently large there exist $(x_\alpha, y_\alpha) \in \Ome \times \Ome$ which maximize $(x,y) \mapsto u_\epsilon(x) - v(y) - \frac{\alpha}{2} | x - y |^2$, as the maxima cannot be attained at the boundary. Appealing to Theorem 3.2, (3.9) and (3.10) of \cite{Crandall_Ishii_Lions92}, there are
\[
(\alpha (x_\alpha - y_\alpha), X) \in \overline{J}^{2,+} u_\epsilon(x_\alpha), \qquad (\alpha (x_\alpha - y_\alpha), Y) \in \overline{J}^{2,-} v(y_\alpha)
\]
such that $X \le Y$. \blue{Therefore
\begin{equation} \label{eq:contradict}
\begin{array}{rl}
0 & = H(X, f(x_\alpha)) - H(Y, f(y_\alpha)) + H(Y, f(y_\alpha)) - H(X, f(x_\alpha)) \\[1mm]
& \leq \frac{f(y_\alpha) - f(x_\alpha)}{d} - \epsilon,
\end{array}
\end{equation}
where we used $H(X, f(x_\alpha)) \le - \varepsilon$ and $H(Y, f(y_\alpha)) \ge 0$ and
\begin{align*}
H(Y, f(y_\alpha)) - H(X, f(x_\alpha)) & \le \sup_{B \in \bS_1} \Bigl( -B:(Y-X) + (f(y_\alpha) - f(x_\alpha)) \sqrt[d]{\det B} \Bigr)\\
& \le (f(y_\alpha) - f(x_\alpha)) \sup_{B \in \bS_1} \sqrt[d]{\det B}.
\end{align*}
Owing to the continuity of $f$ we find $f(y_\alpha) - f(x_\alpha) \to 0$ as $\alpha \to \infty$, so that \cref{eq:contradict} is a contradiction. Hence $u_\epsilon(x) \le v(x)$ for small $\epsilon > 0$ and $x \in \oOme$.}
\end{proof}

\begin{remark}[General boundary conditions and convexity] It is a straightforward exercise to show that we can impose the more general (possibly nonlinear) boundary conditions $(p,r,x) \to B(p,r,x)$ in the viscosity sense in \cref{EMA} and \cref{EHJB}, where the new argument $p$ takes the role of a gradient, and retain equal solution sets.
 
We also observe that the proof of equivalence does not need convexity of the domain~$\Ome$. We note however a close relationship between boundary conditions, comparison and convexity in \cite{Gutierrez01} and also in the \cref{sec:convergence} below, where we study convergence of numerical methods.
\end{remark}

\section{Monotone semi-Lagrangian methods} \label{sec:SLmethods}

In \cref{sec:MA_HJB} we prove that the Monge-Amp\`ere problem \cref{EMA} has a Bellman reformulation \cref{EHJB} in the viscosity sense. This equivalence opens a route for developing numerical methods for \cref{EMA} via \cref{EHJB}. There are major advantages in pursuing this approach. 

\begin{enumerate}[(a)]
\item In \cref{EMA} convexity is built into the boundary value problem as a constraint, cf.~\cref{def:cviscsol}, that is difficult to maintain at the discrete level. In contrast, the convexity of the solution is not enforced as a constraint in \cref{EHJB}. Instead, it arises implicitly from the structure of Bellman operator.
\item For monotone discretizations of Bellman equations there is a well-established framework of semi-smooth Newton methods, also known as Howard's algorithm \cite{Howard60}, which guarantee global superlinear convergence when solving the finite-dimensional equation. These methods have a successful track record for large-scale computations. Howard's algorithm also ensures existence and uniqueness of numerical solutions.
\item The treatment of the degenerate case $f(x) = 0$ is naturally incorporated in the converge proof and does not lead to complications in the analysis.
\item The literature on numerical methods for Bellman-type equations is in various aspects richer than that for Monge-Amp\`ere-type equations, for instance because of the connection to stochastic control problems. As a result, one can use or adapt the numerical methods for Bellman-type equations to solve Monge-Amp\`ere-type equations.
\end{enumerate}

In order to permit unstructured meshes we employ continuous \blue{linear} finite element spaces. \blue{Let $\cT_h$ denote a shape-regular triangular or tetrahedral partition, where $h$ is its mesh function. This means that
\begin{align} \label{eq:meshfun}
x \in T \text{ where } T \in \cT_h \qquad \implies \qquad h(x) = \diam(T).
\end{align}
On element boundaries $h(x)$ is equal to the diameter of the largest element neighboring it; so we could say that $h$ is the upper semicontinuous function with domain $\overline{\Omega}$ satisfying \cref{eq:meshfun}. We abbreviate $\| h \|_{L^\infty(\Omega)}$ by $\hmax$.} We denote by $\cN_h^I$ and $\cN_h^B$ respectively the interior and boundary grid points of $\cT_h$ and set $\cN_h:=\cN_h^I\cup \cN_h^B$. The union of elements, denoted $\Ome_h$, is called the computational domain. Because $\Ome$ is strictly convex, $\Ome_h$ cannot be equal to $\Ome$. We require that $\Ome_h$ approximates $\Ome$ in the sense that $\cN_h^B \subset \partial \Ome$ and $\Ome_h \subset \Ome$. 

\begin{figure}
\begin{center}
\begin{tikzpicture}[
  line width=0.2mm,
  >=stealth',
  dot/.style = {
   draw,
   fill = white,
   circle,
   inner sep = 0pt,
   minimum size = 4pt
  }
 ]
 \foreach \i [evaluate={\ii=int(\i-2);}] in {-3,-1,1,3}{
  \coordinate (\i) at (\i,0.15*\i*\i);
 }
 \foreach \i [evaluate={\ii=int(\i-2);}] in {-2,0,2}{
  \coordinate (\i) at (\i,1.0+0.15*\i*\i);
 }
 \coordinate (m34) at (-3.4,1.8375);
 \coordinate (m21) at (-2.1,1.8375);
 \coordinate (m19) at (-1.8,1.8375);
 \coordinate (m01) at (-0.4,1.8375);
 \coordinate (01) at (0.4,1.8375);
 \coordinate (19) at (1.8,1.8375);
 \coordinate (21) at (2.1,1.8375);
 \coordinate (34) at (3.4,1.8375);
 \coordinate (n1) at (-2,0.75);
 \coordinate (n2) at (-2.3,0.25);
 \filldraw[gray!10] (-3) -- (-1) -- (1) -- (3) -- (34) -- (m34);
 \draw[domain=-3.5:3.5, samples=50, very thick] plot (\x, {0.15*\x*\x});
 \draw[gray] (-3) -- (-2) -- (-1) -- (0) -- (1) -- (2) -- (3);
 \draw (m34) -- (-3) -- (-1) -- (1) -- (3) -- (34);
 \draw[gray] (-2) -- (0) -- (2);
 \draw[gray] (-2) -- (m21);
 \draw[gray] (-2) -- (m19);
 \draw[gray] (0) -- (m01);
 \draw[gray] (0) -- (01);
 \draw[gray] (2) -- (19);
 \draw[gray] (2) -- (21);
 \draw[->] (n1) -- (n2);
 \node at (-3,0.05) {outer normal vector};
 \node at (1,0.9) {$\Omega_h$};
 \node at (2.6,0.7) {$\partial \Omega$};
\end{tikzpicture}
\caption{$\Ome$ is approximated by $\Ome_h$ so that the nodes on $\partial \Ome_h$ belong to $\partial \Ome$. To extend functions $v : \Ome_h \to \mathbf{R}$ to $\Ome$, we assume that the extended function is constant along the normal coordinates of $\partial \Ome_h$, for $x \in \Ome \setminus \Ome_h$.}
\label{fig:computationaldomain}
\end{center}
\end{figure}
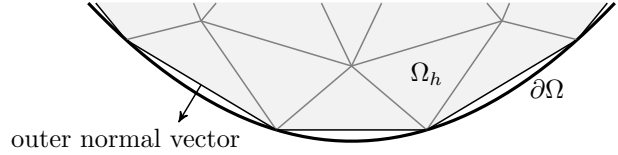

Let $V_h$ denote the space of continuous piecewise linear polynomials over $\cT_h$ and $V_h^0$ be the subspace of $V_h$ consisting of those functions which vanishes at every grid point in $\cN_h^B$. Further, let $\{ \psi_h^j\}_{j=1}^{J_0}$ denote the nodal basis for $V_h^0$ and $\{ \psi_h^j\}_{j=1}^{J}$ denote the nodal basis for $V_h$, where $J_0:=\mbox{card}\bigl(\cN_h^I\bigr)$ and $J:=\mbox{card}\bigl(\cN_h\bigr)$ are the cardinal numbers of $\cN_h^I$ and $\cN_h$, respectively. Often $\psi_h^j$ is called a hat function. In order to study convergence of numerical solutions we need to embed $V_h$ into $B(\Ome)$, i.e.~extend the domain of $v \in V_h$ from $\Ome_h$ to $\Ome$. We shall understand that $v \in V_h$ is extended \blue{as a} constant along the outer normal vectors of $\p \Ome_h$, see \cref{fig:computationaldomain}.  It is not intended that this extension is implemented in numerical codes. 

We first state a basic finite difference formula, which serves as building block for the numerical schemes in this paper. Let $\mathbf{b}\in \mathbf{R}^d$. For smooth $\phi : \mathbf{R}^d\to \mathbf{R}$ there holds for $k > 0$ and $x\in \mathbf{R}^d$
\begin{align} 
\trace \bigl[\bb\bb^T D^2\phi(x) \bigr] &=D^2\phi(x) \bb\cdot\bb = \p_{\bb\bb}^2 \phi(x) \label{eqn:SLscalar} \\
&= \frac{\phi(x-k\bb) -2\phi(x) + \phi(x+k\bb)}{k^2} + O(k^2). \nonumber
\end{align}
The proof of \cref{eqn:SLscalar} for $\phi\in C^4(\mathbf{R}^d)$ follows readily from an application of Taylor's formula. We omit the details.

For a $d\times d$ real valued matrix $\bsigma$, let $\bsigma=(\bsig_1, \bsig_2,\cdots, \bsig_d)$ with $\bsig_j\in \mathbf{R}^d$ denoting the $j$th column vector of $\bsigma$. Let $\bsigma^T$ be the transpose of $\bsigma$ and let $\blambda$ be a diagonal matrix with $\lambda_j$ in the $j$th position of the diagonal. Using \cref{eqn:SLscalar} we immediately get for all $x \in \mathbf{R}^d$:
\begin{align}\label{eqn:SLvector}
\bsigma \blambda \bsigma^T:D^2\phi(x) &=\trace \big[\bsigma \blambda \bsigma^T D^2\phi(x) \bigr] 
=\sum_{j=1}^d \trace \big[\lambda_j \sigma_j\sigma_j^T D^2\phi(x) \bigr] 
\\
&=\sum_{j=1}^d \lambda_j \frac{\phi(x-k\bsig_j) -2\phi(x) + \phi(x+k\bsig_j)}{k^2} + O(k^2) 
, \nonumber
\end{align}
where $A:B$ stands for the Frobenius inner product between two matrices $A$ and $B$. It is an important feature of \cref{eqn:SLvector} that the explicit finite difference discretization of mixed derivatives is avoided in order to build monotonicity into the scheme.

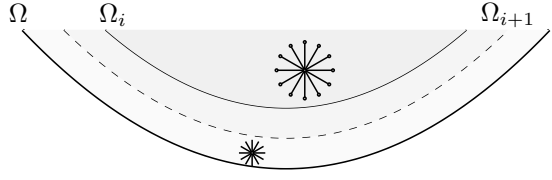
\begin{figure}
\begin{center}
\begin{tikzpicture}[
  line width=0.2mm,
  >=stealth',
  dot/.style = {
   draw,
   fill = white,
   circle,
   inner sep = 0pt,
   minimum size = 1pt
  }
 ]
 \draw[domain=-3.5:3.5, samples=50, very thick] plot (\x, {0.15*\x*\x});
 \fill[gray!4] (-3.5, 1.8375) -- (-3.465, 1.80093) -- (-3.43, 1.76474) -- (-3.395, 1.7289) -- (-3.36, 1.69344) -- (-3.325, 1.65834) -- (-3.29, 1.62361) -- (-3.255, 1.58925) -- (-3.22, 1.55526) -- (-3.185, 1.52163) -- (-3.15, 1.48837) -- (-3.115, 1.45548) -- (-3.08, 1.42296) -- (-3.045, 1.3908) -- (-3.01, 1.35901) -- (-2.975, 1.32759) -- (-2.94, 1.29654) -- (-2.905, 1.26585) -- (-2.87, 1.23554) -- (-2.835, 1.20558) -- (-2.8, 1.176) -- (-2.765, 1.14678) -- (-2.73, 1.11794) -- (-2.695, 1.08945) -- (-2.66, 1.06134) -- (-2.625, 1.03359) -- (-2.59, 1.00621) -- (-2.555, 0.979204) -- (-2.52, 0.95256) -- (-2.485, 0.926284) -- (-2.45, 0.900375) -- (-2.415, 0.874834) -- (-2.38, 0.84966) -- (-2.345, 0.824854) -- (-2.31, 0.800415) -- (-2.275, 0.776344) -- (-2.24, 0.75264) -- (-2.205, 0.729304) -- (-2.17, 0.706335) -- (-2.135, 0.683734) -- (-2.1, 0.6615) -- (-2.065, 0.639634) -- (-2.03, 0.618135) -- (-1.995, 0.597004) -- (-1.96, 0.57624) -- (-1.925, 0.555844) -- (-1.89, 0.535815) -- (-1.855, 0.516154) -- (-1.82, 0.49686) -- (-1.785, 0.477934) -- (-1.75, 0.459375) -- (-1.715, 0.441184) -- (-1.68, 0.42336) -- (-1.645, 0.405904) -- (-1.61, 0.388815) -- (-1.575, 0.372094) -- (-1.54, 0.35574) -- (-1.505, 0.339754) -- (-1.47, 0.324135) -- (-1.435, 0.308884) -- (-1.4, 0.294) -- (-1.365, 0.279484) -- (-1.33, 0.265335) -- (-1.295, 0.251554) -- (-1.26, 0.23814) -- (-1.225, 0.225094) -- (-1.19, 0.212415) -- (-1.155, 0.200104) -- (-1.12, 0.18816) -- (-1.085, 0.176584) -- (-1.05, 0.165375) -- (-1.015, 0.154534) -- (-0.98, 0.14406) -- (-0.945, 0.133954) -- (-0.91, 0.124215) -- (-0.875, 0.114844) -- (-0.84, 0.10584) -- (-0.805, 0.0972037) -- (-0.77, 0.088935) -- (-0.735, 0.0810337) -- (-0.7, 0.0735) -- (-0.665, 0.0663337) -- (-0.63, 0.059535) -- (-0.595, 0.0531037) -- (-0.56, 0.04704) -- (-0.525, 0.0413437) -- (-0.49, 0.036015) -- (-0.455, 0.0310537) -- (-0.42, 0.02646) -- (-0.385, 0.0222337) -- (-0.35, 0.018375) -- (-0.315, 0.0148837) -- (-0.28, 0.01176) -- (-0.245, 0.00900375) -- (-0.21, 0.006615) -- (-0.175, 0.00459375) -- (-0.14, 0.00294) -- (-0.105, 0.00165375) -- (-0.07, 0.000735) -- (-0.035, 0.00018375) -- (4.44089*10^-16, 2.95823*10^-32) -- (0.035, 0.00018375) -- (0.07, 0.000735) -- (0.105, 0.00165375) -- (0.14, 0.00294) -- (0.175, 0.00459375) -- (0.21, 0.006615) -- (0.245, 0.00900375) -- (0.28, 0.01176) -- (0.315, 0.0148838) -- (0.35, 0.018375) -- (0.385, 0.0222338) -- (0.42, 0.02646) -- (0.455, 0.0310538) -- (0.49, 0.036015) -- (0.525, 0.0413438) -- (0.56, 0.04704) -- (0.595, 0.0531038) -- (0.63, 0.059535) -- (0.665, 0.0663338) -- (0.7, 0.0735) -- (0.735, 0.0810338) -- (0.77, 0.088935) -- (0.805, 0.0972038) -- (0.84, 0.10584) -- (0.875, 0.114844) -- (0.91, 0.124215) -- (0.945, 0.133954) -- (0.98, 0.14406) -- (1.015, 0.154534) -- (1.05, 0.165375) -- (1.085, 0.176584) -- (1.12, 0.18816) -- (1.155, 0.200104) -- (1.19, 0.212415) -- (1.225, 0.225094) -- (1.26, 0.23814) -- (1.295, 0.251554) -- (1.33, 0.265335) -- (1.365, 0.279484) -- (1.4, 0.294) -- (1.435, 0.308884) -- (1.47, 0.324135) -- (1.505, 0.339754) -- (1.54, 0.35574) -- (1.575, 0.372094) -- (1.61, 0.388815) -- (1.645, 0.405904) -- (1.68, 0.42336) -- (1.715, 0.441184) -- (1.75, 0.459375) -- (1.785, 0.477934) -- (1.82, 0.49686) -- (1.855, 0.516154) -- (1.89, 0.535815) -- (1.925, 0.555844) -- (1.96, 0.57624) -- (1.995, 0.597004) -- (2.03, 0.618135) -- (2.065, 0.639634) -- (2.1, 0.6615) -- (2.135, 0.683734) -- (2.17, 0.706335) -- (2.205, 0.729304) -- (2.24, 0.75264) -- (2.275, 0.776344) -- (2.31, 0.800415) -- (2.345, 0.824854) -- (2.38, 0.84966) -- (2.415, 0.874834) -- (2.45, 0.900375) -- (2.485, 0.926284) -- (2.52, 0.95256) -- (2.555, 0.979204) -- (2.59, 1.00622) -- (2.625, 1.03359) -- (2.66, 1.06134) -- (2.695, 1.08945) -- (2.73, 1.11794) -- (2.765, 1.14678) -- (2.8, 1.176) -- (2.835, 1.20558) -- (2.87, 1.23554) -- (2.905, 1.26585) -- (2.94, 1.29654) -- (2.975, 1.32759) -- (3.01, 1.35902) -- (3.045, 1.3908) -- (3.08, 1.42296) -- (3.115, 1.45548) -- (3.15, 1.48838) -- (3.185, 1.52163) -- (3.22, 1.55526) -- (3.255, 1.58925) -- (3.29, 1.62362) -- (3.325, 1.65834) -- (3.36, 1.69344) -- (3.395, 1.7289) -- (3.43, 1.76474) -- (3.465, 1.80093) -- (3.5, 1.8375) -- cycle;
 \draw[domain=-2.95163:2.95163, samples=50, dashed] plot (\x, {0.165*\x*\x+0.4});
 \fill[gray!8] (-2.95163, 1.8375) -- (-2.92211, 1.80889) -- (-2.8926, 1.78057) -- (-2.86308, 1.75254) -- (-2.83356, 1.7248) -- (-2.80405, 1.69734) -- (-2.77453, 1.67017) -- (-2.74502, 1.64329) -- (-2.7155, 1.6167) -- (-2.68598, 1.59039) -- (-2.65647, 1.56437) -- (-2.62695, 1.53864) -- (-2.59743, 1.5132) -- (-2.56792, 1.48804) -- (-2.5384, 1.46317) -- (-2.50889, 1.43859) -- (-2.47937, 1.4143) -- (-2.44985, 1.39029) -- (-2.42034, 1.36657) -- (-2.39082, 1.34314) -- (-2.3613, 1.32) -- (-2.33179, 1.29714) -- (-2.30227, 1.27457) -- (-2.27276, 1.25229) -- (-2.24324, 1.2303) -- (-2.21372, 1.20859) -- (-2.18421, 1.18717) -- (-2.15469, 1.16604) -- (-2.12517, 1.1452) -- (-2.09566, 1.12464) -- (-2.06614, 1.10437) -- (-2.03662, 1.08439) -- (-2.00711, 1.0647) -- (-1.97759, 1.04529) -- (-1.94808, 1.02617) -- (-1.91856, 1.00734) -- (-1.88904, 0.9888) -- (-1.85953, 0.970544) -- (-1.83001, 0.952575) -- (-1.80049, 0.934894) -- (-1.77098, 0.9175) -- (-1.74146, 0.900394) -- (-1.71195, 0.883575) -- (-1.68243, 0.867044) -- (-1.65291, 0.8508) -- (-1.6234, 0.834844) -- (-1.59388, 0.819175) -- (-1.56436, 0.803794) -- (-1.53485, 0.7887) -- (-1.50533, 0.773894) -- (-1.47582, 0.759375) -- (-1.4463, 0.745144) -- (-1.41678, 0.7312) -- (-1.38727, 0.717544) -- (-1.35775, 0.704175) -- (-1.32823, 0.691094) -- (-1.29872, 0.6783) -- (-1.2692, 0.665794) -- (-1.23968, 0.653575) -- (-1.21017, 0.641644) -- (-1.18065, 0.63) -- (-1.15114, 0.618644) -- (-1.12162, 0.607575) -- (-1.0921, 0.596794) -- (-1.06259, 0.5863) -- (-1.03307, 0.576094) -- (-1.00355, 0.566175) -- (-0.974038, 0.556544) -- (-0.944522, 0.5472) -- (-0.915005, 0.538144) -- (-0.885489, 0.529375) -- (-0.855973, 0.520894) -- (-0.826456, 0.5127) -- (-0.79694, 0.504794) -- (-0.767424, 0.497175) -- (-0.737907, 0.489844) -- (-0.708391, 0.4828) -- (-0.678875, 0.476044) -- (-0.649359, 0.469575) -- (-0.619842, 0.463394) -- (-0.590326, 0.4575) -- (-0.56081, 0.451894) -- (-0.531293, 0.446575) -- (-0.501777, 0.441544) -- (-0.472261, 0.4368) -- (-0.442744, 0.432344) -- (-0.413228, 0.428175) -- (-0.383712, 0.424294) -- (-0.354196, 0.4207) -- (-0.324679, 0.417394) -- (-0.295163, 0.414375) -- (-0.265647, 0.411644) -- (-0.23613, 0.4092) -- (-0.206614, 0.407044) -- (-0.177098, 0.405175) -- (-0.147581, 0.403594) -- (-0.118065, 0.4023) -- (-0.0885489, 0.401294) -- (-0.0590326, 0.400575) -- (-0.0295163, 0.400144) -- (0., 0.4) -- (0.0295163, 0.400144) -- (0.0590326, 0.400575) -- (0.0885489, 0.401294) -- (0.118065, 0.4023) -- (0.147582, 0.403594) -- (0.177098, 0.405175) -- (0.206614, 0.407044) -- (0.23613, 0.4092) -- (0.265647, 0.411644) -- (0.295163, 0.414375) -- (0.324679, 0.417394) -- (0.354196, 0.4207) -- (0.383712, 0.424294) -- (0.413228, 0.428175) -- (0.442744, 0.432344) -- (0.472261, 0.4368) -- (0.501777, 0.441544) -- (0.531293, 0.446575) -- (0.56081, 0.451894) -- (0.590326, 0.4575) -- (0.619842, 0.463394) -- (0.649359, 0.469575) -- (0.678875, 0.476044) -- (0.708391, 0.4828) -- (0.737907, 0.489844) -- (0.767424, 0.497175) -- (0.79694, 0.504794) -- (0.826456, 0.5127) -- (0.855973, 0.520894) -- (0.885489, 0.529375) -- (0.915005, 0.538144) -- (0.944522, 0.5472) -- (0.974038, 0.556544) -- (1.00355, 0.566175) -- (1.03307, 0.576094) -- (1.06259, 0.5863) -- (1.0921, 0.596794) -- (1.12162, 0.607575) -- (1.15114, 0.618644) -- (1.18065, 0.63) -- (1.21017, 0.641644) -- (1.23968, 0.653575) -- (1.2692, 0.665794) -- (1.29872, 0.6783) -- (1.32823, 0.691094) -- (1.35775, 0.704175) -- (1.38727, 0.717544) -- (1.41678, 0.7312) -- (1.4463, 0.745144) -- (1.47582, 0.759375) -- (1.50533, 0.773894) -- (1.53485, 0.7887) -- (1.56436, 0.803794) -- (1.59388, 0.819175) -- (1.6234, 0.834844) -- (1.65291, 0.8508) -- (1.68243, 0.867044) -- (1.71195, 0.883575) -- (1.74146, 0.900394) -- (1.77098, 0.9175) -- (1.80049, 0.934894) -- (1.83001, 0.952575) -- (1.85953, 0.970544) -- (1.88904, 0.9888) -- (1.91856, 1.00734) -- (1.94808, 1.02617) -- (1.97759, 1.04529) -- (2.00711, 1.0647) -- (2.03662, 1.08439) -- (2.06614, 1.10437) -- (2.09566, 1.12464) -- (2.12517, 1.1452) -- (2.15469, 1.16604) -- (2.18421, 1.18717) -- (2.21372, 1.20859) -- (2.24324, 1.2303) -- (2.27276, 1.25229) -- (2.30227, 1.27457) -- (2.33179, 1.29714) -- (2.3613, 1.32) -- (2.39082, 1.34314) -- (2.42034, 1.36657) -- (2.44985, 1.39029) -- (2.47937, 1.4143) -- (2.50889, 1.43859) -- (2.5384, 1.46317) -- (2.56792, 1.48804) -- (2.59743, 1.5132) -- (2.62695, 1.53864) -- (2.65647, 1.56437) -- (2.68598, 1.59039) -- (2.7155, 1.6167) -- (2.74502, 1.64329) -- (2.77453, 1.67017) -- (2.80405, 1.69734) -- (2.83356, 1.7248) -- (2.86308, 1.75254) -- (2.8926, 1.78057) -- (2.92211, 1.80889) -- (2.95163, 1.8375) -- cycle; 
 \draw[domain=-2.4:2.4, samples=50] plot (\x, {0.18*\x*\x+0.8});
 \fill[gray!12] (-2.4, 1.8368) -- (-2.376, 1.81617) -- (-2.352, 1.79574) -- (-2.328, 1.77553) -- (-2.304, 1.75551) -- (-2.28, 1.73571) -- (-2.256, 1.71612) -- (-2.232, 1.69673) -- (-2.208, 1.67755) -- (-2.184, 1.65857) -- (-2.16, 1.63981) -- (-2.136, 1.62125) -- (-2.112, 1.6029) -- (-2.088, 1.58475) -- (-2.064, 1.56682) -- (-2.04, 1.54909) -- (-2.016, 1.53157) -- (-1.992, 1.51425) -- (-1.968, 1.49714) -- (-1.944, 1.48024) -- (-1.92, 1.46355) -- (-1.896, 1.44707) -- (-1.872, 1.43079) -- (-1.848, 1.41472) -- (-1.824, 1.39886) -- (-1.8, 1.3832) -- (-1.776, 1.36775) -- (-1.752, 1.35251) -- (-1.728, 1.33748) -- (-1.704, 1.32265) -- (-1.68, 1.30803) -- (-1.656, 1.29362) -- (-1.632, 1.27942) -- (-1.608, 1.26542) -- (-1.584, 1.25163) -- (-1.56, 1.23805) -- (-1.536, 1.22467) -- (-1.512, 1.21151) -- (-1.488, 1.19855) -- (-1.464, 1.18579) -- (-1.44, 1.17325) -- (-1.416, 1.16091) -- (-1.392, 1.14878) -- (-1.368, 1.13686) -- (-1.344, 1.12514) -- (-1.32, 1.11363) -- (-1.296, 1.10233) -- (-1.272, 1.09124) -- (-1.248, 1.08035) -- (-1.224, 1.06967) -- (-1.2, 1.0592) -- (-1.176, 1.04894) -- (-1.152, 1.03888) -- (-1.128, 1.02903) -- (-1.104, 1.01939) -- (-1.08, 1.00995) -- (-1.056, 1.00072) -- (-1.032, 0.991704) -- (-1.008, 0.982892) -- (-0.984, 0.974286) -- (-0.96, 0.965888) -- (-0.936, 0.957697) -- (-0.912, 0.949714) -- (-0.888, 0.941938) -- (-0.864, 0.934369) -- (-0.84, 0.927008) -- (-0.816, 0.919854) -- (-0.792, 0.912908) -- (-0.768, 0.906168) -- (-0.744, 0.899636) -- (-0.72, 0.893312) -- (-0.696, 0.887195) -- (-0.672, 0.881285) -- (-0.648, 0.875583) -- (-0.624, 0.870088) -- (-0.6, 0.8648) -- (-0.576, 0.85972) -- (-0.552, 0.854847) -- (-0.528, 0.850181) -- (-0.504, 0.845723) -- (-0.48, 0.841472) -- (-0.456, 0.837428) -- (-0.432, 0.833592) -- (-0.408, 0.829964) -- (-0.384, 0.826542) -- (-0.36, 0.823328) -- (-0.336, 0.820321) -- (-0.312, 0.817522) -- (-0.288, 0.81493) -- (-0.264, 0.812545) -- (-0.24, 0.810368) -- (-0.216, 0.808398) -- (-0.192, 0.806636) -- (-0.168, 0.80508) -- (-0.144, 0.803732) -- (-0.12, 0.802592) -- (-0.096, 0.801659) -- (-0.072, 0.800933) -- (-0.048, 0.800415) -- (-0.024, 0.800104) -- (0., 0.8) -- (0.024, 0.800104) -- (0.048, 0.800415) -- (0.072, 0.800933) -- (0.096, 0.801659) -- (0.12, 0.802592) -- (0.144, 0.803732) -- (0.168, 0.80508) -- (0.192, 0.806636) -- (0.216, 0.808398) -- (0.24, 0.810368) -- (0.264, 0.812545) -- (0.288, 0.81493) -- (0.312, 0.817522) -- (0.336, 0.820321) -- (0.36, 0.823328) -- (0.384, 0.826542) -- (0.408, 0.829964) -- (0.432, 0.833592) -- (0.456, 0.837428) -- (0.48, 0.841472) -- (0.504, 0.845723) -- (0.528, 0.850181) -- (0.552, 0.854847) -- (0.576, 0.85972) -- (0.6, 0.8648) -- (0.624, 0.870088) -- (0.648, 0.875583) -- (0.672, 0.881285) -- (0.696, 0.887195) -- (0.72, 0.893312) -- (0.744, 0.899636) -- (0.768, 0.906168) -- (0.792, 0.912908) -- (0.816, 0.919854) -- (0.84, 0.927008) -- (0.864, 0.934369) -- (0.888, 0.941938) -- (0.912, 0.949714) -- (0.936, 0.957697) -- (0.96, 0.965888) -- (0.984, 0.974286) -- (1.008, 0.982892) -- (1.032, 0.991704) -- (1.056, 1.00072) -- (1.08, 1.00995) -- (1.104, 1.01939) -- (1.128, 1.02903) -- (1.152, 1.03888) -- (1.176, 1.04894) -- (1.2, 1.0592) -- (1.224, 1.06967) -- (1.248, 1.08035) -- (1.272, 1.09124) -- (1.296, 1.10233) -- (1.32, 1.11363) -- (1.344, 1.12514) -- (1.368, 1.13686) -- (1.392, 1.14878) -- (1.416, 1.16091) -- (1.44, 1.17325) -- (1.464, 1.18579) -- (1.488, 1.19855) -- (1.512, 1.21151) -- (1.536, 1.22467) -- (1.56, 1.23805) -- (1.584, 1.25163) -- (1.608, 1.26542) -- (1.632, 1.27942) -- (1.656, 1.29362) -- (1.68, 1.30803) -- (1.704, 1.32265) -- (1.728, 1.33748) -- (1.752, 1.35251) -- (1.776, 1.36775) -- (1.8, 1.3832) -- (1.824, 1.39886) -- (1.848, 1.41472) -- (1.872, 1.43079) -- (1.896, 1.44707) -- (1.92, 1.46355) -- (1.944, 1.48024) -- (1.968, 1.49714) -- (1.992, 1.51425) -- (2.016, 1.53157) -- (2.04, 1.54909) -- (2.064, 1.56682) -- (2.088, 1.58475) -- (2.112, 1.6029) -- (2.136, 1.62125) -- (2.16, 1.63981) -- (2.184, 1.65857) -- (2.208, 1.67755) -- (2.232, 1.69673) -- (2.256, 1.71612) -- (2.28, 1.73571) -- (2.304, 1.75551) -- (2.328, 1.77553) -- (2.352, 1.79574) -- (2.376, 1.81617) -- (2.4, 1.8368) -- cycle;
 \node at (-3.5,2.03) {$\Omega_{\phantom{i}}$};
 \node at (2.95163,2.03) {$\Omega_{i+1}$};
 \node at (-2.3,2.03) {$\Omega_{i}$};
 \begin{scope}[shift={(-0.45,0.2)},scale=0.07]
 \foreach \x in {0,30,60,90,120,150,180,210,240,270,300,330} {\draw (0,0) to (\x:2.5);};
 \node [dot] at (0.0:0.0) {};
 \end{scope}
 \begin{scope}[shift={(0.25,1.3)},scale=0.15]
 \foreach \x in {0,30,60,90,120,150,180,210,240,270,300,330} {\draw (0,0) to (\x:2.5); \node [dot] at (\x:2.5) {};};
 \node [dot] at (0.0:0.0) {};
 \end{scope}
\end{tikzpicture}
\caption{\blue{The $\Omega_i$ form a covering of $\Omega$. On each of the $\Ome_i$ the Wasow-Motzkin consistency condition `$h/k \to 0$' is implemented uniformly; near the boundary this is not enforced as local stencils are rescaled so that they do not extend out of the computational domain---illustrated by two cartoon stenils in the figure.}}
\label{fig:subdomains}
\end{center}
\end{figure}

\blue{The choice of $k$ depends on $h$ and $x$: 
\begin{enumerate}[(a)]
\item It is known as Wasow-Motzkin theorem \cite[Theorem 1]{WasowMotzkin53} that in order to achieve consistency with equations like \cref{EMAeq,EHJBeq} simultaneously with monotonicity, the mesh size $h$ has to decrease locally strictly faster than the stencil size $k$, see also \cite{Kocan95}. Therefore we expect $k$ to decrease as the mesh size $h$ shrinks, but within this `$h/k \to 0$' limitation. {\em In other words, the Wasow-Motzkin theorem implies that any monotone consistent method has to be a wide stencil scheme.}
\item Observe that if $\phi\in C^4(\oOme)$ then \cref{eqn:SLvector} remains valid as long as the stencil size $k$ is chosen small enough so that the stencil does not extend out of the domain. Hence near the boundary the size of $k$ needs to be reduced to the size of $h$ for $x \in \cN_h^I$. This makes $k$ dependent on $x$.
\end{enumerate}
A specific choice for $k$ is given in \cref{ex:k} below. In general, condition (b) is reflected by the requirement that
\[
k : \; L^\infty(\Omega) \times \Ome \to (0,\infty), \; (h,x) \mapsto k(h,x)
\]
is a function such that $x-k(h,x) \, \bsig_j$ and $x+k(h,x) \, \bsig_j$ are in $\oOme_h$ for all mesh functions $h$ and $x \in \Ome_h$ and $\sigma_j$. Condition (a) is in conflict with this as (b) implies that $h$ cannot decrease faster than $k$ near $\partial \Omega$. Therefore we shall impose Wasow-Motzkin limitation uniformly only on the subsets 
\[
\Omega_i = \left\{ x \in \Omega : {\rm distance}(x,\p\Ome) > {\textstyle \frac{1}{i}} \right\},
\]
illustrated in \cref{fig:subdomains}, see also the related \cref{stencils}. Thus on each $\Omega_i$ we require
\begin{align} \label{eq:cons_scaling}
\sup_{x \in \Ome_i} \frac{h(x)}{k(h,x)} \to 0 \qquad \text{as} \qquad \hmax \to 0,
\end{align}
recalling that $\hmax$ is the largest diameter of an element of the mesh.
\red{Furthermore, we shall assume that on each $\Omega_i$ the stencil size $k$ is eventually a constant function: for every $i \in \bN$ there is an $h'$ so that $x \mapsto k(h,x)$ is a constant function on $\Omega_i$ whenever $\| h \|_\infty < h'$.} Moreover, we assume that the stencil size $k$ shrinks uniformly, meaning that on the whole domain~$\Omega$
\begin{align} \label{eq:krefinement}
\sup_{x \in \Ome} k(h,x) \to 0 \qquad \text{as} \qquad \hmax \to 0.
\end{align}}
\begin{remark} \label{ex:k}
\blue{As prototypical choice for $k$ we have in mind that
\begin{align} \label{samplek}
k(h,x) = \min \{ \kappa(\hmax), {\rm distance}(x,\p\Ome) \} \qquad \forall \, x \in \Ome
\end{align}
for some $\kappa : (0,\infty) \to (0,\infty)$ with $\xi / \kappa(\xi) \to 0$ and $\kappa(\xi) \to 0$ as $\xi \to 0$, e.g.~\mbox{$\kappa(\xi)=\sqrt{\xi}$}. Observe that once $\kappa(\hmax) < \frac{1}{i}$ then $k = \kappa(\hmax)$ is constant on the restriction to $\Omega_i$ because there $\kappa(\hmax) < {\rm distance}(x,\p\Ome)$. Since the calculation of ${\rm distance}(x,\p\Ome)$ can be computationally expensive, one should in practice implement an approximation of \cref{samplek} satisfying \cref{eq:cons_scaling} and \cref{eq:krefinement}.}
\end{remark}

To discretize the linear operators
\begin{align} \label{eq:LB}
\phi \mapsto -B:D^2\phi(x) + f \sqrt[d]{\det B},
\end{align}
which are found under the supremum of \cref{HJB_operator}, we choose factorizations $B = \bsigma \blambda \bsigma^T$ for each $B \in \bS_1$. More precisely, we consider some compact set
\[
\bF \subset \mathbf R^{d \times d} \times \{ A \in \mathbf R^{d \times d} : A \text{ diagonal} \}
\]
such that the mapping
\begin{align} \label{factorisation} 
\bF \to \bS_1, \; (\bsigma, \blambda) \mapsto \bsigma \, \blambda \, \bsigma^T
\end{align}
is bijective. \red{Moreover we assume that all $\blambda$ have the same trace:
\begin{align} \label{eq:trace}
\exists \, C > 0 \, \forall \, (\bsigma, \blambda) \in \bF: \, \trace \blambda = C.
\end{align}}
The existence of such $\bsigma$ and $\blambda$ follows from the symmetry of~$B$. We remark that strictly speaking \cref{factorisation} only needs to be surjective for the subsequent analysis; however, without injectivity the notation becomes more cumbersome as more than one factorization represents a single $B$. We remark that our analysis also extends to direction dependent $k = k(h,x,\sigma_j)$, owing to the compactness of $\bF$.

At this point there is considerable flexibility in the selection of $\bF$. We discuss concrete choices in \cref{sec-ps}, after examining the well-posedness of the discrete equations in \cref{sec-dwp} and the convergence of numerical solutions in \cref{sec:convergence}.

The approximation of \cref{eq:LB} is the mapping $L_h^B : \; \bR \times B(\oOme) \to B(\oOme)$, where for any $\phi\in B(\oOme)$ the value $L_h^B(s, \phi)(x_i)$ at internal node $x_i \in \cN_h^I$ is set to be
\begin{align} \label{eq:L_pw_disc}
- \sum_{j=1}^d \lambda_j \frac{\phi(x_i-k\bsig_j) -2 s + \phi(x_i+k\bsig_j)}{k^2} + f(x_i) \sqrt[d]{\det B},
\end{align}
where $\bsigma = \bsigma(B)$ and $\blambda = \blambda(B)$ come from inversion of \cref{factorisation}. Recall that $k = k(h,x)$ is chosen so that $x-k(h,x) \, \bsig_j$ and $x+k(h,x) \, \bsig_j$ are in $\oOme_h$. Also observe how $s$ takes, in comparison with \cref{eqn:SLvector}, the place of $\phi(x_i)$. The auxiliary variable $s$ allows us to express the monotonicity of the scheme efficiently in the language of the Barles-Souganidis framework \cite{Barles_Souganidis91}, on which we model our proof of convergence. The value $L_h^B(s, \phi)(x_i)$ for boundary nodes $x_i \in \cN_h^B$ is
\begin{align} \label{eq:L_pw_disc_bc}
s - g(x_i).
\end{align}
At nodes $x_i \in \cN_h$, the Bellman operator $H$ is represented approximately by
\[
H_h(s, \phi)(x_i) = \sup_{B \in \bS_1} L_h^B(s, \phi)(x_i).
\]
For the remaining $x \in \Ome_h \setminus \cN_h$ the value of $H_h(s, \phi)(x)$ is defined by piecewise linear interpolation of the nodal values, so that we have a mapping
\begin{align} \label{eq:H_pw_disc}
H_h : \; \bR \times B(\oOme) \to B(\oOme),
\end{align}
upon constant extension in the normal direction for $x \in \oOme \setminus \Ome_h$, recall \cref{fig:computationaldomain}.

Finally, our numerical scheme for \cref{EHJB} is defined as seeking $u_h\in V_h$ such that
\begin{align} \label{eq:disc}
H_h(u_h(x_i), u_h)(x_i) = 0 \qquad \forall \,x_i \in \cN_h.
\end{align}

\section{Well-posedness of the discrete equations}\label{sec-dwp}

A common technique to show the well-posedness of a nonlinear system such as \cref{eq:disc} is to formulate a fixed point argument akin to a pseudo-time Euler scheme \cite{Debrabant_Jakobsen12,Oberman08}. However, to take advantage of the monotone discretization of the Bellman equation, we use instead Howard's algorithm \cite{BMZ09,Howard60} to establish the existence and uniqueness of numerical solutions. This algorithm, being globally superlinearly converging, is also used to compute the numerical solutions of our numerical experiments in \cref{sec:experiments}.

Let 
\[
\bB = (B_1, B_2, \ldots, B_{J_0}) = (\bsigma^{(1)} \blambda^{(1)} (\bsigma^{(1)})^T, \bsigma^{(2)} \blambda^{(2)} (\bsigma^{(2)})^T, \ldots, \bsigma^{(J_0)} \blambda^{(J_0)} (\bsigma^{(J_0)})^T)
\]
be an element of $\bS_1^{J_0}$. Then $\rL_h^\bB : \; V_h \to V_h$ discretizes $\phi \mapsto -B_i :D^2\phi(x)$ at the internal nodes as
\begin{align} \label{eq:VL_pw_disc}
\rL_h^\bB(\phi)(x_i) = \begin{cases}
- \sum_{j=1}^d \lambda^{(i)}_j \frac{\phi(x_i-k\bsig^{(i)}_j) -2 \phi(x_i) + \phi(x_i+k\bsig^{(i)}_j)}{k^2} & : \; x_i \in \cN_h^I,\\[1mm]
\phi(x_i) & : \; x_i \in \cN_h^B.
\end{cases}
\end{align}
Similarly, we set
\begin{align} \label{eq:VF_pw_disc}
\rF_h^\bB(x_i) = \begin{cases}
f(x_i) \sqrt[d]{\det B_i} & : \; x_i \in \cN_h^I,\\[1mm]
- g(x_i) & : \; x_i \in \cN_h^B.
\end{cases}
\end{align}
For the remaining $x \in \overline{\Ome_h} \setminus \cN_h$ the values of $\rL_h^\bB(\phi)(x)$ and $\rF_h^\bB(x)$ are defined by piecewise linear interpolation of the nodal values. It is worthwhile to bring the differences between $\rL_h^\bB$ in \cref{eq:VL_pw_disc} and $L_h^B$ in \cref{eq:L_pw_disc} to mind. While the former has the right structure for the \blue{finite-dimensional} analysis of Howard's algorithm, the latter mirrors the Barles--Souganidis formulation \blue{with the additional argument $s$ to examine the monotonicity property efficiently}.

\begin{lemma} \label{lin_nonneg}
Let $\bB \in \bS_1^{J_0}$ and $h > 0$. If $\rL_h^\bB v \le 0$ then $v$ attains its maximum at a boundary node. Moreover, the representation of the mapping $\rL_h^\bB$ as \blue{a matrix} $A$, using the linear finite element hat functions as basis, is an invertible M-matrix.
\end{lemma}

\begin{proof}
Let $X$ be the set of nodes where $v$ attains its maximum. Suppose that $X$ consists only of internal nodes, i.e.~$X \subset \cN_h^I$, and that $\rL_h^\bB v \le 0$ holds. Let $C(X)$ be the convex hull of $X$. Let $x_i$ be an extreme point of $C(X)$; it is clear that such $x_i$ exists, not least by the Krein-Milman theorem, and that it is a node. For each $1 \le j \le d$, the value $v(x_i \pm k\bsig_j^{(i)})$ is a weighted average of the nodal values of $v$ at the vertices $x_\ell$ of the finite element which contains $x_i \pm k\bsig_j^{(i)}$. It follows from $\rL_h^\bB v \le 0$ that $v$ is equal to $v(x_i)$ at all those nodes $x_\ell$ whenever~$\lambda_j \neq 0$, noting that there is at least one non-zero $\lambda_j$. Thus $x_i \pm k\bsig_j^{(i)} \in C(X)$, which contradicts that $x_i$ is an extreme point. Hence $X$ must contain a boundary node.

Suppose that $A v = 0$. Then $v$ attains its maximum and, considering the argument of the previous paragraph for $- v$, its minimum on the boundary. As the restriction of $A$ to boundary nodes is an identity map, it follows that $v = 0$. Hence $A$ is invertible.

Owing to \cref{eq:VL_pw_disc}, \blue{$A \in \mathbf{R}^{n \times n}$ and $a_{ij} \le 0$ for all $i \neq j$ and with $n=J_0$. Moreover}, $A + \epsilon \, {\rm Id}$ is strictly diagonally dominant for each $\epsilon > 0$. Therefore such $A + \epsilon \, {\rm Id}$ are M-matrices \cite[Theorem (2.3) with $(M_{35})$ on p.137]{BermanPlemmons} and hence $A$ is a M-matrix \cite[Theorem (2.3) with $(D_{15})$ on p.135]{BermanPlemmons}.  
\end{proof} 

It follows directly from the construction of the discrete Hamiltonian that the numerical scheme, defined in \cref{eq:disc}, is equivalent to
\begin{align*} 
\sup_{\bB \in \bS_1^{J_0}} \bigl[ \rL_h^\bB(u_h)(x_i) + \rF_h^\bB(x_i) \bigr] = 0 \qquad \forall \, x_i \in \cN_h.
\end{align*}
\blue{For the solution of \cref{eq:disc} we use \cref{alg:Howard}, known as Howard's method.}

\begin{algorithm}
\caption{Howard's method}
\label{alg:Howard}
\begin{algorithmic}[1]
\STATE{Select an arbitrary $\bB_0 \in \bS_1^{J_0}$}
\FOR{$\ell \in \bN$}
\STATE{Let $v^\ell$ be the solution of the affine equation $\rL_h^\bB(v^\ell) + \rF_h^\bB= 0$.}
\STATE{Set $\bB_{\ell+1} = {\rm argmax}_{\bB \in \bS_1^{J_0}} \,\bigl[ \rL_h^\bB(v^\ell) + \rF_h^\bB \bigr]$}
\ENDFOR
\end{algorithmic}
\end{algorithm}

\begin{theorem} \label{thm:disk_wp}
Let $h > 0$ and assume that $\bF$ is compact and \cref{factorisation} bijective. Then for every $g \in B(\p\Ome)$ there exists a unique numerical solution $u_h \in V_h$ of \cref{eq:disc}. Moreover, the sequence $(v^\ell)_\ell$ generated by Howard's algorithm converges monotonically decreasing and superlinearly to $u_h$ as $\ell \to \infty$.
\end{theorem}

\begin{proof}
Due to the bijectivity of \cref{factorisation} we may consider $\bF$ instead of $\bS_1$ as the set of controls. Clearly the mappings $(\bsigma, \blambda) \mapsto \rL_h^\bB$ are $(\bsigma, \blambda) \mapsto \rF_h^\bB$ are continuous. The monotonicity condition of \cite{BMZ09} is verified in above \cref{lin_nonneg}. The result now follows from Theorem 2.1 of \cite{BMZ09}, noting that the substitution of the $\min$ in (1.1) of \cite{BMZ09} by $\max$ reverses the direction of the monotone convergence. The superlinear rate follows from Theorem 3.4 of \cite{BMZ09}.
\end{proof}

Observe that this well-posedness result for the discrete Bellman problem does not require convexity of the domain---the proof of \cref{lin_nonneg} remains valid for non-convex $\Omega$, even though then possibly $C(X) \not\subset \Omega_h$, where $C(X)$ is defined in the proof of \cref{lin_nonneg}.

\section{Convergence analysis}\label{sec:convergence}

Comparison principles are a central component of the theory of viscosity solutions. With Perron's method they are commonly used to show existence of solutions. For the analysis of numerical methods, the Barles--Souganidis framework, which we loosely follow in this section, provides a convergence argument based on comparison of subsolutions and supersolutions.

Dirichlet boundary conditions pose here a particular difficulty. The strong comparison principle underlying the original Barles--Souganidis proof requires comparison of semi-continuous subsolutions and supersolutions, which obey boundary conditions in the viscosity sense. Yet, general degenerate elliptic equations usually only satisfy comparison of semi-continuous functions with pointwise Dirichlet conditions or comparison of continuous functions with viscosity Dirichlet conditions \cite[Section 7.C]{Crandall_Ishii_Lions92}. The combination as in the Barles--Souganidis framework without additional structure assumptions about the boundary value problem does in general not hold.

To resolve this mismatch we verify that the upper and lower semicontinuous envelopes of the numerical solutions satisfy the boundary conditions pointwise, at which point the Barles--Souganidis argument becomes in its essential steps available. In fact, this \cref{classicalBC} is the only place in our analysis where the convexity of the domain is used, being aware that a Barles-Souganidis argument is a proof of existence and uniqueness of viscosity solutions.

\blue{We introduce}
\begin{align} \label{eq:S_HJB}
S : \; \mathbf{R}_+\times \overline{\Ome}\times \mathbf{R} \times B(\overline{\Ome})\to \mathbf{R}, \; (\blue{\hmax},x,s,\phi) \mapsto \blue{\hmax} \, H_h(s,\phi)(x),
\end{align}
\blue{to match precisely the structure of the solution operator in (2.1) of \cite{Barles_Souganidis91}.}

\begin{lemma}\label{monotonicity}
The mapping $S$ is monotone in the sense that
\[
S(\blue{\hmax},x,s,u)\leq S(\blue{\hmax},x,s,v) \qquad \text{if} \qquad u\geq v 
\]
for all $\blue{\hmax} > 0$, $x \in \overline{\Ome}$, $s \in \mathbf{R}$ and $u,v\in B(\overline{\Ome})$.
\end{lemma}

\begin{proof}
This follows directly from \cref{eq:L_pw_disc} and \cref{eq:L_pw_disc_bc}.
\end{proof}

For the proof of stability we construct a comparison function $\zeta$. This $\zeta$ will subsequently also be used to show that the envelopes of the numerical solutions obey the boundary conditions in the pointwise sense, cf.~\cref{classicalBC}. 

\begin{lemma}\label{stability}
The mapping $S$ is stable: there exists a $h$-independent constant $C>0$ such that
\begin{align} \label{stabconst}
\|u_h\|_{L^\infty(\oOme)} \leq C
\end{align}
for $u_h$ given by \cref{eq:disc}. Furthermore, let $p \in \mathbf R^d$ and choose
\begin{align} \label{comparisoncoef}
M \ge \| f \|_{L^\infty(\Ome)} \, \max_{B \in \bS_1} \sqrt[d]{\det B} = \blue{\frac{\| f \|_{L^\infty(\Ome)}}{d}}.
\end{align}
Let $\zeta(x;M,p) = \textstyle \frac{M}{2} | x - p |^2$ and $I_h$ be the nodal interpolant onto $V_h$. Then, for all $h > 0$, the function $u_h - I_h \zeta$ (resp.~$u_h + I_h \zeta$) attains its minimum (resp.~maximum) over $\oOme$ at a boundary node. 
\end{lemma}

\begin{proof}
Let 
\[
\bB' = {\rm argmax}_{\bB \in \bS_1^{J_0}} \,\bigl[ \rL_h^\bB(u_h) + \rF_h^\bB \bigr].
\]
Then $S(\blue{\hmax},x_i,u_h(x_i),u_h) / \blue{\hmax} = \rL_h^{\bB'}(u_h)(x_i) + \rF_h^{\bB'}(x_i)$ at $x_i \in \cN_h$.

To derive a bound on $u_h$ from below, let $\zeta(x) = \zeta(x; M, p)$ be as in the statement of the theorem. \blue{Observe that for any internal node $x_i$, also near the boundary $\partial \Omega_h$,}
\begin{align*} 
\sum_{j=1}^d \lambda_j^{(i)} \frac{\zeta(\blue{x_i}-k\bsig^{(i)}_j) -2\zeta(\blue{x_i}) + \zeta(\blue{x_i}+k\bsig^{(i)}_j)}{k^2} & = \sum_{j=1}^d \lambda_j \partial_{\sigma^{(i)},\sigma^{(i)}}^2 \zeta(x_i)\\
& \hspace{-2.5cm} \stackrel{\cref{eqn:SLvector}}{=} B'_i : D^2 \zeta(x_i) = M \, (B'_i : {\rm Id}) = M \, ( {\rm tr} \, B'_i) = M.
\end{align*}
Because of the convexity of $\zeta$ we know that
\[
I_h \zeta(\blue{x_i}-k\bsig^{(i)}_j) \ge \zeta(\blue{x_i}-k\bsig^{(i)}_j), \qquad I_h \zeta(\blue{x_i}+k\bsig^{(i)}_j) \ge \zeta(\blue{x_i}+k\bsig^{(i)}_j)
\]
and consequently, \blue{since $I_h \rho(x_i)=\rho(x_i)$ as $x_i$ is a node},
\[
\sum_{j=1}^d \lambda_j^{(i)} \frac{I_h \zeta(\blue{x_i}-k\bsig^{(i)}_j) -2 I_h \zeta(\blue{x_i}) + I_h \zeta(\blue{x_i}+k\bsig^{(i)}_j)}{k^2} \ge M.
\]
Hence, with $N \in \mathbf R$,
\[
\rL_h^{\bB'}(I_h \zeta - u_h - N)(x_i) \le \begin{cases} 
\textstyle f(x_i) \sqrt[d]{\det B'_i} - M & : \text{ if } x_i \in \cN_h^I,\\
\textstyle \frac{M}{2} | x_i - p |^2 - g(x_i) - N & : \text{ if } x_i \in \cN_h^B.
\end{cases}
\]
As, for $N$ large, $\rL_h^{\bB'}(I_h \zeta - u_h - N) \le 0$ on $\oOme$ it \blue{follows from \cref{lin_nonneg}} that $I_h \zeta - u_h - N$ and equally $I_h \zeta - u_h$ attain their maximum at a boundary node $x_i$. Thus, for $x \in \oOme$,
\begin{align*}
- u_h(x) & \le \| I_h \zeta \|_{L^\infty(\Ome)} + \| I_h \zeta - u_h \|_{L^\infty(\partial \Ome)} \le 2 \, \| \zeta \|_{L^\infty(\Ome)} + \| g \|_{L^\infty(\partial \Ome)}
\end{align*}
gives an $h$-independent bound on $u_h$ from below.

Now to the bound from above. As for large $N$ we have $\rL_h^{\bB'}(u_h - N) \le 0$ on $\oOme$ it follows from \cref{lin_nonneg} that $u_h$ attains its maximum at a boundary node $x_i$, where $u_h(x_i) = g(x_i)$. Thus $u_h$ is bounded from above by $g$. It is also clear that the maximizer of $u_h + I_h \zeta(x)$ is attained on $\partial \Ome_h$ and $\partial \Ome$, in fact for any $M \ge 0$.
\end{proof} 

Our consistency condition \cref{conscond} differs from \cite{Barles_Souganidis91} in that we require $x \in \Ome$ instead of $x \in \oOme$, however see also \cref{classicalBC}. Indeed we would not expect our scheme to be consistent as in the Barles--Souganidis framework, predicted by the results in \cite{WasowMotzkin53,Kocan95} due to the violation of \cref{eq:cons_scaling} in the vicinity of the boundary. One would assume that any numerical method cropping a wide stencil near $\p\Ome$ is incompatible with the original Barles--Souganidis framework, because the viscosity boundary conditions used there require the consistent monotone discretization of both the boundary \blue{operator} and the differential operator at all $x \in \p\Ome$ for all test functions $\phi$, see (7.9)--(7.10) of \cite{Crandall_Ishii_Lions92} and (2.4) of \cite{Barles_Souganidis91}.

\begin{lemma}\label{consistency}
The mapping $S$ of \cref{eq:S_HJB} is consistent in the sense that for all $x\in \Ome$ and $\phi \in C^4(\Omega)$ there hold
\begin{align} \label{conscond} \begin{array}{rcl}
\displaystyle \limsup_{\substack{\blue{\hmax} \to 0\\ y\to x\\ \xi\to 0}}
\frac{ S(\blue{\hmax}, y, \phi(y)+\xi, \phi+\xi) }{\blue{\hmax}} 
\leq H(D^2\phi(x), f(x)),\\[2mm]
\displaystyle \liminf_{\substack{\blue{\hmax} \to 0\\ y\to x\\ \xi\to 0}}
\frac{ S(\blue{\hmax}, y, \phi(y)+\xi, \phi+\xi) }{\blue{\hmax}} 
\geq H(D^2\phi(x), f(x)).
\end{array} \end{align}

\end{lemma}

\begin{proof}
\blue{There is an $i$ such that $x \in \Ome_i$. Also $x \in \Ome_h$ for $\blue{\hmax}$ sufficiently small. Recall that $\Ome_i$ and $\Ome_h$ are open. Hence we may restrict our attention to $y \in \Ome_i \cap \Ome_h$. Let $\blue{\hmax}$ be small enough such that $k(h,x) + \blue{\hmax} < {\rm dist}(y,\p \Ome_h \cup \p \Ome_i)$. The numerical operator fully expanded is
\begin{align*} 
& \, \frac{1}{\hmax} \, S(\hmax,y,\phi(y) + \xi, \phi + \xi)\\
= & \, I_h \Bigl( x_i \mapsto \sup_{B \in \bS_1} - \sum_{j = 1}^d \lambda_j \frac{\phi(x_i + k \, \sigma_j) - 2 \phi(y) + \phi(x_i - k \, \sigma_j)}{k^2} + f(x_i) \sqrt[d]{\det B} \Bigr)(y), \nonumber
\end{align*}
where the interpolation operator $I_h$ acts on a mapping which assigns to internal nodes $x_i$ real values arising from the supremum over the finite-difference-like terms and the determinant term. The finite-difference-like terms depend on the $y$, however, and therefore are not a central differences at this point. Finally, the interpolation operator returns an interpolating finite element function, which is evaluated at the very same $y$. The $\xi$, appearing in the Barles--Souganidis formulation \cref{conscond} of consistency, cancels itself out immediately.

To prove consistency we first restore the central differences. We denote the $B$ maximising above at node $x_i$ by $B_i = \bsigma^{(i)} \blambda^{(i)} (\bsigma^{(i)})^T$. Then
\begin{align*} 
& \, \frac{1}{\hmax} \, S(\hmax,y,\phi(y) + \xi, \phi + \xi)\\
= & \, I_h \Bigl( x_i \mapsto - \sum_{j = 1}^d \lambda^{(i)}_j \frac{\phi(x_i + k \, \sigma^{(i)}_j) - 2 \phi(y) + \phi(x_i - k \, \sigma^{(i)}_j)}{k^2} + f(x_i) \sqrt[d]{\det B} \Bigr)(y)\\
= & \, I_h \Bigl( x_i \mapsto - \sum_{j = 1}^d \lambda^{(i)}_j \frac{\phi(x_i + k \, \sigma^{(i)}_j) - 2 \phi(x_i) + \phi(x_i - k \, \sigma^{(i)}_j)}{k^2} + f(x_i) \sqrt[d]{\det B} \Bigr)(y)\\
& + I_h \Bigl( x_i \mapsto \frac{2 \, \trace (\blambda^{(i)}) \, \bigl( \phi(y) - \phi(x_i) \bigr)}{k^2} \Bigr)(y)\\
\le & \, I_h \Bigl( x_i \mapsto \sup_{B \in \bS_1} - \sum_{j = 1}^d \lambda_j \frac{\phi(x_i + k \, \sigma_j) - 2 \phi(x_i) + \phi(x_i - k \, \sigma_j)}{k^2} + f(x_i) \sqrt[d]{\det B} \Bigr)(y)\\
& + I_h \Bigl( x_i \mapsto \frac{2 \, \trace (\blambda^{(i)}) \, \bigl( \phi(y) - \phi(x_i) \bigr)}{k^2} \Bigr)(y).
\end{align*}
Denoting the maximising $B$ in the last display at node $x_i$ by $\bar{B}_i = \bar{\bsigma}^{(i)} \bar{\blambda}^{(i)}(\bar{\bsigma}^{(i)})^T$, we obtain similarly
\begin{align*} 
& \, \frac{1}{\hmax} \, S(\hmax,y,\phi(y) + \xi, \phi + \xi)\\
\ge & \, I_h \Bigl( x_i \mapsto - \sum_{j = 1}^d \bar{\lambda}^{(i)}_j \frac{\phi(x_i + k \, \bar{\sigma}^{(i)}_j) - 2 \phi(y) + \phi(x_i - k \, \bar{\sigma}^{(i)}_j)}{k^2} + f(x_i) \sqrt[d]{\det B} \Bigr)(y)\\
= & \, I_h \Bigl( x_i \mapsto \sup_{B \in \bS_1} - \sum_{j = 1}^d \lambda_j \frac{\phi(x_i + k \, \sigma_j) - 2 \phi(x_i) + \phi(x_i - k \, \sigma_j)}{k^2} + f(x_i) \sqrt[d]{\det B} \Bigr)(y)\\
& + I_h \Bigl( x_i \mapsto \frac{2 \, \trace (\bar\blambda^{(i)}) \, \bigl( \phi(y) - \phi(x_i) \bigr)}{k^2} \Bigr)(y).
\end{align*}
\red{Because of \cref{eq:trace} we conclude that the last two inequalities are in fact equalities and that the traces of $\blambda$ and $\bar\blambda$ may be taken out of $I_h$.} For the test functions $\phi$ 
\begin{align*} 
\sup_{z \in \Ome_i} \; \sup_{B \in \bS_1} \; \sup_{1 \le j \le d} \; \left| \frac{\phi(z + k \, \sigma_j) - 2 \phi(z) + \phi(z - k \, \sigma_j) }{k^2} - \partial^2_{\sigma_j,\sigma_j} \phi(z) \right|
\end{align*}
is of the order $O(\sup_{z \in \Ome_i} k^2(h,z))$. Thence,
\begin{align*}
& \, I_h \Bigl( x_i \mapsto \sup_{B \in \bS_1} - \sum_{j = 1}^d \lambda_j \frac{\phi(x_i + k \, \sigma_j) - 2 \phi(x_i) + \phi(x_i - k \, \sigma_j)}{k^2} + f(x_i) \sqrt[d]{\det B} \Bigr)(y)\\
= & \, I_h \Bigl( x_i \mapsto \sup_{B \in \bS_1} - \sum_{j = 1}^d \lambda_j \partial^2_{\sigma_j,\sigma_j} \phi(x_i) + f(x_i) \sqrt[d]{\det B} \Bigr)(y) + O(\sup_{z \in \Ome_i} k^2(h,z))\\
\to & \, H(D^2 \phi(x), f(x))
\end{align*}
as $\hmax \to 0$ and $y \to x$, since $O(\sup_{z \in \Ome_i} k^2(h,z)) \to 0$ as $\blue{\hmax} \to 0$. 

Finally we show that $I_h \bigl( x_i \mapsto ( \phi(y) - \phi(x_i) ) / k^2 \bigr)(y) \to 0$ as $\hmax \to 0$. 
Recall that
\[
\| \phi - I_h \phi \|_{L^\infty(\Ome_h)} \le C \, \blue{\hmax}^2 \| \phi \|_{W^{2,\infty}(\Ome_h)},
\]
see \cite[Corollary 1.109]{ErnGuermond}.  \red{There is a neighborhood $N$ of $y$ so that eventually} all elements containing $y$ belong to $N$ and \red{$k$ is constant on $N$.} Thus, for $\hmax$ small enough,
\begin{align*} 
I_h \Bigl( x_i \mapsto \frac{\phi(y) - \phi(x_i)}{k^2} \Bigr)(y) = \frac{\phi(y) - (I_h \phi) (y)}{k(h,y)^2} \to 0,
\end{align*}
as $\hmax \to 0$ due to \cref{eq:cons_scaling}.}
\end{proof}

\begin{figure}
\begin{center}
\begin{tikzpicture}[
  line width=0.2mm,
  >=stealth',
  dot/.style = {
   draw,
   fill = white,
   circle,
   inner sep = 0pt,
   minimum size = 4pt
  }
 ]
 \coordinate (O) at (0,0);
 \draw[->] (0,0) -- (11,0) coordinate[label = {below:$\nabla P$}] (xmax);
 \draw[-] (0,-1.5) -- (0,3) coordinate[label = {right:${\rm ker} P$}] (ymax);
 \draw plot[smooth] coordinates {(0.7,-1.5) (0,0) (1,1.5) (4,3)};
 \draw [dashed] (2,0) arc [radius=8, start angle=180, end angle= 190];
 \draw [dashed] (2,0) arc [radius=8, start angle=180, end angle= 160];
 \draw[-,dotted] (2.3,2.2) -- (2.3,0);
 \node [dot] at (2,0) {};
 \node [below] at (1.8,0) {$y_\ell$};
 \node [dot] at (2.3,0) {};
 \node [below] at (2.5,0) {$y_\infty$};
 \node [dot] at (1.2,0) {};
 \draw[<->] (0,-0.2) -- (1.2,-0.2);
 \node [below] at (0.6,-0.2) {$2 K$};
 \node [left] at (2.3,2.3) {$y$};
 \node [dot] at (2.3,2.2) {};
 \node [left] at (0,0) {$x$};
 \node [dot] at (0,0) {};
 \node [below] at (10,0) {$p_\ell$};
 \node [dot] at (10,0) {};
 \node [right] at (2.05,-1.2) {$\partial B_\ell$};
 \node [below] at (3.2,2.6) {$\partial \Omega$};
\end{tikzpicture}
\caption{The constant $K$ can be found for sufficiently large $\ell$ because $y_\ell$ converges to $y_\infty$ as $\ell \to \infty$, where $B_\ell$ denotes the ball with center $p_\ell$ and radius $|y - p_\ell|$.}
\label{findK}
\end{center}
\end{figure}
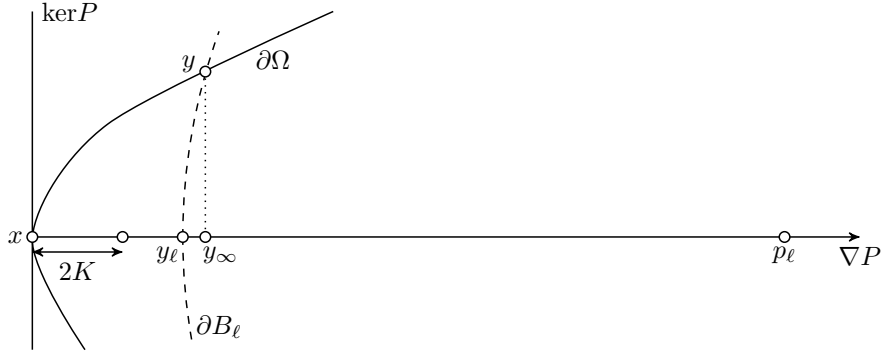

We define, with $x,y \in \oOme$ and $h > 0$,
\begin{align} \label{envelopes}
\overline{u}(x):=\limsup_{\substack{y\to x\\ \blue{\hmax} \to 0}} u_h(y) \qquad\mbox{and}\qquad
\underline{u}(x):=\liminf_{\substack{y\to x\\ \blue{\hmax} \to 0}} u_h(y). 
\end{align}
The following lemma confirms that $\overline{u} \in \USC(\oOme)$ and $\underline{u} \in \LSC(\oOme)$ are consistent with the pointwise Dirichlet conditions at the boundary.

\begin{lemma} \label{classicalBC}
\blue{Let $\Ome$ be a strictly convex domain, then we} have $\overline{u}(x) = \underline{u}(x) = g(x)$ 
for all $x \in \partial \Ome$.
\end{lemma}

\begin{proof}
We show that $\underline{u}$ satisfies the pointwise boundary conditions on $\partial \Ome$. The proof for $\overline{u}$ is analogous. Fix $x \in \partial \Omega$. As $\Ome$ is convex there exists an affine mapping $P : \mathbf R^d \to \mathbf R$ such that
\[
( \oOme \setminus \{x\} ) \subset \{ y \in \mathbf R^d : P y > 0\} \qquad \text{and} \qquad P x = 0.
\]
For $\ell \in \mathbf N$ let $p_\ell = x + \ell \nabla P$, noting that $\nabla P$ is an inner normal vector of $\Ome$. Again we use $\zeta_\ell(x) = \zeta(x; M, p_\ell)$ with $M$ satisfying \cref{comparisoncoef}. We denote by $q_\ell$ the minimizer of $g - \zeta_\ell$ over $\p\Ome$.

Due to compactness of $\p\Ome$ the sequence $(q_\ell)_\ell$ has a converging subsequence with a limit $y \in \p\Ome$. If $y \neq x$ it follows from the strict convexity that there is a constant $K > 0$ such that $|y - p_\ell| + 2 K \le |x - p_\ell|$ for all large $\ell$, see \cref{findK}. But then, for $|q_\ell - y| < K$,
\begin{align} \label{quad_inequality}
\zeta_\ell(x) \blue{= \frac{M}{2} | x- p_\ell|^2} > \frac{M}{2} \left( | q_\ell - p_\ell| + K \right)^2 = \zeta_\ell(q_\ell) + M K | q_\ell - p_\ell| + \frac{M K^2}{2}.
\end{align}
Because $C < M K | q_\ell - p_\ell|$ for large $\ell$ with $C$ as in \cref{stabconst}, \cref{quad_inequality} contradicts that $q_\ell$ is a minimizer. Hence 
\begin{align} \label{limit_to_boundary}
\lim_{\ell \to \infty} q_\ell = x.
\end{align}
Consider a sequence $(y_i, \blue{\hmax_i})_{i \in \mathbf N}$ with $\lim_{i \to \infty} (y_i,\blue{\hmax_i}) = (x,0)$. Then, for all $\ell \in \mathbf N$,
\begin{align*}
\liminf_{i \to \infty} u_{h_i}(y_i) & = \lim_{i \to \infty} I_{h_i} \zeta_\ell(y_i) 
+ \liminf_{i \to \infty} \bigl[ u_{h_i}(y_i) - I_{h_i} \zeta_\ell(y_i) \bigr]\\
& \ge \zeta_\ell(x) + \liminf_{i \to \infty} \inf_{y \in \p\Ome} \bigl[ u_{h_i}(y) - I_{h_i} \zeta_\ell(y) \bigr]\\
& \ge \zeta_\ell(x) + \liminf_{i \to \infty} \inf_{y \in \p\Ome} \bigl[ g(y) - \zeta_\ell(y) \bigr]\\
& = \zeta_\ell(x) + g(q_\ell) - \zeta_\ell(q_\ell),
\end{align*}
where we used that $u_{h_i} - I_{h_i} \zeta_\ell$ attains its minimum at a node on the boundary, cf.~\cref{stability}. Together with \cref{limit_to_boundary} we have
\[
\liminf_{i \to \infty} u_{h_i}(y_i) \ge g(x).
\]
As this inequality holds for all sequences $(y_i, \blue{\hmax_i})_{i \in \mathbf N}$ converging to $(x,0)$, we have $\underline{u} \ge g$ on $\p \Ome$. The opposite inequality follows by choosing sequences with $y_i = x$.
\end{proof}

We are ready to state the main result of this paper. 

\begin{theorem}\label{thm:main}
Let $\Ome$ be a strictly convex domain. Assume that $f \in C(\Ome)$ with $f\ge 0$ and $g\in C(\p\Ome)$. Then, as $\blue{\hmax} \to 0$, the solutions $u_h$ of \cref{eq:disc} converge uniformly to a function \blue{$u \in C(\oOme)$}, which is the unique viscosity solution on the set of convex functions of the Monge--Amp\`ere problem \cref{EMAeq}. This $u$ is also the unique viscosity solution of the Bellman problem \cref{EHJBeq} and it satisfies the Dirichlet boundary conditions pointwise.
\end{theorem}

\begin{proof}
We have consistency for $\phi \in C^4(\Omega)$. It follows directly from the proof of Theorem 2.1 in \cite{Barles_Souganidis91} that for all $x \in \Ome$
\[
H(D^2\phi(x), f(x)) \le 0\qquad \text{(}H(D^2\phi(x), f(x)) \ge 0 \text{)}
\]
whenever $\underline{u} - \phi$ (respectively $\overline{u} - \phi$) attains a local maximum (minimum) at $x$. The result carries over \cite[p.57]{Crandall_Ishii_Lions92} to test functions $\phi \in C^2(\Ome)$ so that~$\underline{u}$ and $\overline{u}$ are super- and subsolutions of \cref{EHJBeq}. 

Now \cref{thm:comparison_principle} and \cref{classicalBC} yield $\overline{u} \le \underline{u}$ on $\oOme$. The opposite inequality is clear from the definition of $\overline{u}$ and $\underline{u}$. This, together with \cref{envelopes}, implies the uniform convergence to the unique viscosity solution of the Bellman problem on $\oOme$. Now the result follows from \cref{thm:MA_to_HJB} and \cref{thm:HJB_to_MA}.
\end{proof}

\section{Parameter selection} \label{sec-ps}

It remains to show that a suitable compact set $\bF$ can be found so that \cref{factorisation} is bijective. It turns out that there are several viable candidates.

A natural starting point is the eigen-decomposition $B = Q \Lambda Q^T$ of real symmetric matrices, where where $Q$ is an orthonormal matrix and $\Lambda$ is a diagonal matrix whose entries are the eigenvalues of $B$. Similarly one can use the Cholesky decomposition or the closely related LDL decomposition $B = LDL^T$, where $L$ is a lower unit triangular matrix and $D$ a diagonal matrix. A widely used choice for the discretization of Bellman equations is $B = \bsigma \blambda \bsigma^T = \bsigma \bsigma^T$, that is $\blambda = \Id$, see \cite{Kushner90} and Section 5 of \cite{Debrabant_Jakobsen12}. 

From the implementational point of view it is desirable to keep the set of $\bsigma$ small: While evaluation of $\phi(x_i-k\bsig^{(i)}_j)$ and $\phi(x_i+k\bsig^{(i)}_j)$ in \cref{eq:VL_pw_disc} can be implemented efficiently \cite[Remark 4]{Achdou_Falcone12}, unnecessary evaluations should be avoided, especially if $k \gg h$. In contrast no significant savings arise in \cref{eq:VL_pw_disc} from a small set of $\blambda$. In the numerical experiments in the next section we use therefore the eigen-decomposition of $B$ as in this case the $\sigma_j$ \blue{can be} normalized and multiple $B$ share the same $\bsigma = Q$.

\section{Numerical experiments} \label{sec:experiments}

\begin{figure}
\centerline{
\includegraphics[width=6.0cm]{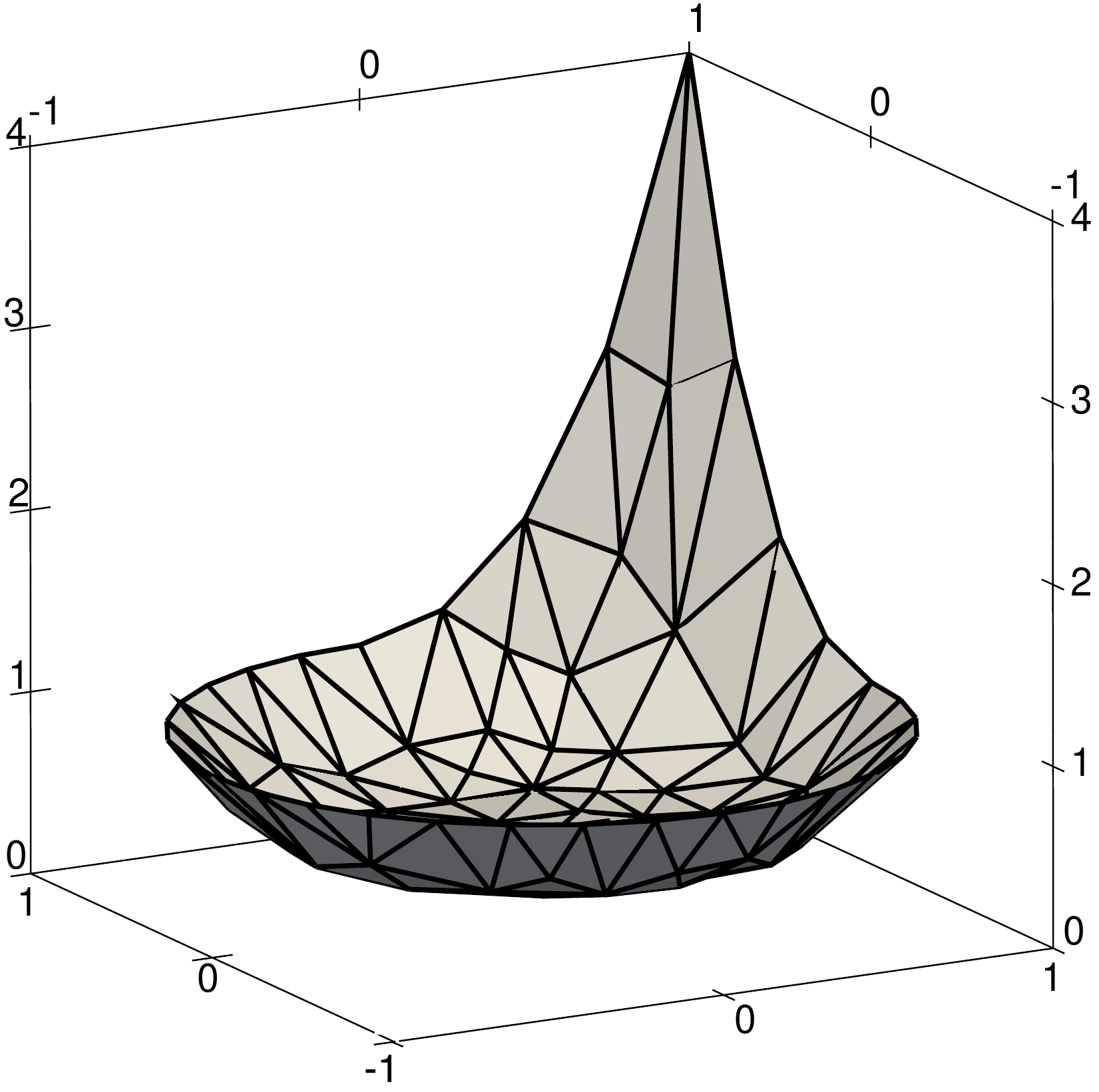}
\includegraphics[width=6.0cm]{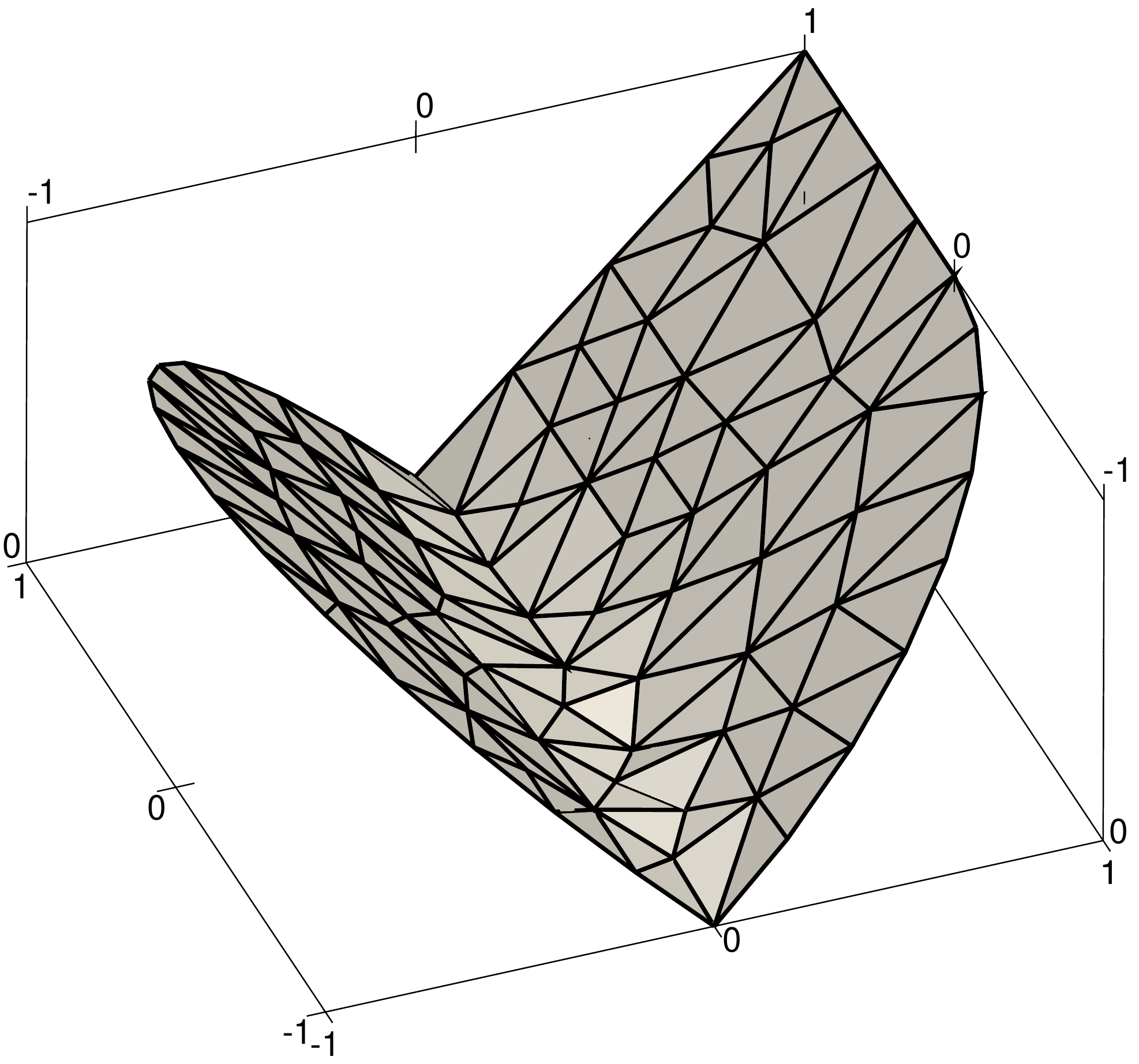}
}
\caption{Solution of the quartic and non-smooth problem on the coarsest mesh.}
\label{fig:numsol}
\end{figure}

In this section we present two 2-d numerical experiments to test the proposed wide-stencil method and Howard's \blue{Newton} solver. The first experiment has the exact smooth solution $u(x)= |x|^4=(x_1^2 + x_2^2)^2$ and the second experiment computes the non-smooth viscosity solution $u(x)=|x_1|$. In both experiments the computational domain is the union of the unit circle and the unit square so that the strict convexity condition is violated in part of the domain:
\[
\Ome = \{ (x,y) \in \mathbf{R}^2 : x^2 + y^2 < 1 \} \cup \{ (x,y) \in \mathbf{R}^2 : 
0 < x,y < 1 \}.
\]
The \blue{quasi-uniform} grid has at the coarsest level $91$ nodes and at the finest level after $7$ uniform refinements $1,214,465$ nodes. The computations were carried out in Python with FEniCS on an Apple iMac computer. The numerical solutions on the coarsest grid are shown in \cref{fig:numsol}.

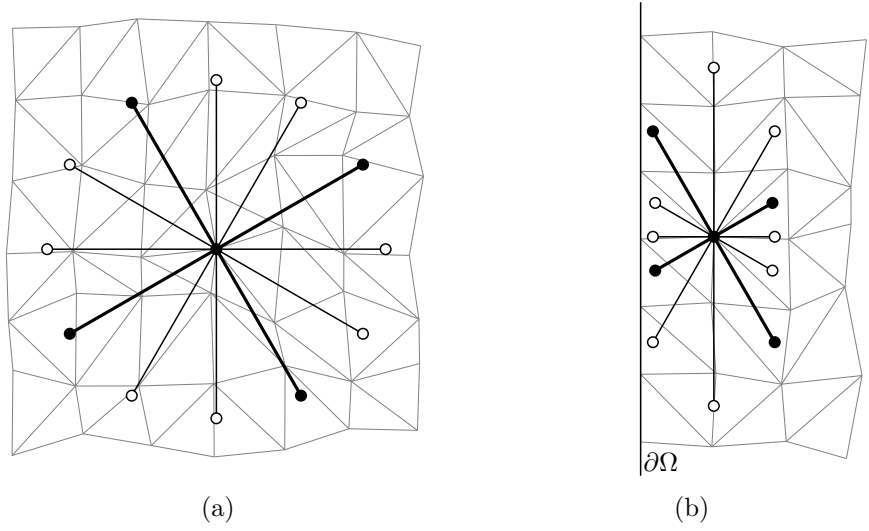
\begin{figure}
\begin{center}
\begin{tikzpicture}[
    line width=0.2mm,
    >=stealth',
    x=0.9cm,y=0.9cm,
    dot/.style = {
      draw,
      fill = white,
      circle,
      inner sep = 0pt,
      minimum size = 4pt
    },
    blackdot/.style = {
      draw,
      fill = black,
      circle,
      inner sep = 0pt,
      minimum size = 4pt
    }
  ]
\foreach \i [evaluate={\ii=int(\i-1);}] in {0,...,6}{
  \foreach \j [evaluate={\jj=int(\j-1);}] in {0,...,6}{
    \coordinate [shift={(\j,\i)}] (n-\i-\j) at (rand*90:1/6+rnd/9);
\ifnum\i>0
  \draw [help lines] (n-\i-\j) -- (n-\ii-\j);
\fi
\ifnum\j>0
  \draw [help lines] (n-\i-\j) -- (n-\i-\jj);
  \ifnum\i>0
    \pgfmathparse{int(rnd>.5)}
    \ifnum\pgfmathresult=0
      \draw [help lines] (n-\i-\j) -- (n-\ii-\jj);
    \else%
      \draw [help lines] (n-\ii-\j) -- (n-\i-\jj);
    \fi%
  \fi
\fi
}}
\begin{scope}[shift=(n-3-3)]
\foreach \x in {0,60,90,150,180,240,270,330} {\draw (0,0) to (\x:2.5); \node [dot] at (\x:2.5) {};};
\foreach \x in {30,120,210,300} {\draw [very thick] (0,0) to (\x:2.5); \node [blackdot] at (\x:2.5) {};};
\node [blackdot] at (0.0:0.0) {};
\end{scope}
\node at (3.2,-1.0) {(a)};

\begin{scope}[xshift=8.5cm,yshift=0]
\foreach \i [evaluate={\ii=int(\i-1);}] in {0,...,6}{
  \foreach \j [evaluate={\jj=int(\j-1);}] in {0,...,3}{
    \coordinate [shift={(\j,\i)}] (n-\i-\j) at (rand*90:\j*1/15+\j*rnd/20);
\ifnum\i>0
  \draw [help lines] (n-\i-\j) -- (n-\ii-\j);
\fi
\ifnum\j>0
  \draw [help lines] (n-\i-\j) -- (n-\i-\jj);
  \ifnum\i>0
    \pgfmathparse{int(rnd>.5)}
    \ifnum\pgfmathresult=0
      \draw [help lines] (n-\i-\j) -- (n-\ii-\jj);
    \else%
      \draw [help lines] (n-\ii-\j) -- (n-\i-\jj);
    \fi%
  \fi
\fi
}}
\begin{scope}[shift=(n-3-1)]
\draw (0:-0.9) to (0:0.9); \node [dot] at (0:0.9) {}; \node [dot] at (0:-0.9) {};
\draw [very thick] (30:-1.0) to (30:1.0); \node [blackdot] at (30:1.0) {}; \node [blackdot] at (30:-1.0) {};
\draw (60:-1.8) to (60:1.8); \node [dot] at (60:1.8) {}; \node [dot] at (60:-1.8) {};
\draw (90:-2.5) to (90:2.5); \node [dot] at (90:2.5) {}; \node [dot] at (90:-2.5) {};
\draw [very thick] (120:-1.8) to (120:1.8); \node [blackdot] at (120:1.8) {}; \node [blackdot] at (120:-1.8) {};
\draw (150:-1.0) to (150:1.0); \node [dot] at (150:1.0) {}; \node [dot] at (150:-1.0) {};
\node [blackdot] at (0.0:0.0) {};
\end{scope}
\draw [-] (0.0,-0.5) -- (0.0,6.5);
\node at (0.3,-0.3) {$\partial \Ome$};
\end{scope};
\node at (10.2,-1.0) {(b)};
\end{tikzpicture}

\caption{Plot (a) shows a stencil of the discrete Hamiltonian where the \blue{finite} differences are spaced at angles of $\pi / 6$ and $m$ is about $2.5$. \blue{The black dots mark a single stencil, the white dots stencil positions of other angles.} Plot (b) illustrates how the \blue{finite} differences are rescaled near the boundary to ensure that the stencil does not extend out of the boundary. \blue{We illustrate here how $k = k(h,x,\sigma_j)$ can also be rescaled depending on the direction $\sigma_j$, noting that our analysis easily extends to this case.} }
\label{stencils}
\end{center}
\end{figure}

The compact control set is
\[
\bF = \left( {\rm SO}(2) \times \left\{ \begin{pmatrix} a & 0\\ 0 & 1-a \end{pmatrix}: a \in \left[0,{\textstyle \frac{1}{2}}\right) \right\} \right) \cup \left\{ \left( \Id, \begin{pmatrix} \frac{1}{2} & 0\\ 0 & \frac{1}{2} \end{pmatrix} \right) \right\}.
\]
In order to compute the numerical solutions we discretize the special orthogonal group ${\rm SO}(2)$ by considering only the rotation angles $i \pi / 64$, $i \in \{0, 1, \ldots, 63\}$, see \cref{stencils} (a) for an illustration of angles $i \pi / 6$. The stencil diameter $k$ is, away from the boundary, represented through $k = m \cdot h$ by a fixed \blue{positive factor} $m$ and the (average) mesh size $h$. \blue{Near the boundary, so where $m \cdot h$ is larger than the distance to $\partial \Omega_h$, the stencil is reduced in size to remain within $\Omega_h$, see \cref{stencils} (b).}
 
\begin{figure}
\begin{center}
\begin{tabular}{r|rr|rr|rr}
\multicolumn{7}{c}{quartic problem}\\[1mm]
DoFs & $L^2$-error & $m$ & $L^\infty$-error & $m$ & $H^1$-error & $m$ \\ \hline
91 & $6.92 \cdot 10^{-2}$ & $2$ & $9.26 \cdot 10^{-2}$ & $2$ & $1.86 \cdot 10^{-1}$ & $2$\\
329 & $2.99 \cdot 10^{-2}$ & $2$ & $3.85 \cdot 10^{-2}$ & $2$ & $9.08 \cdot 10^{-2}$ & $2$\\
1,249 & $1.69 \cdot 10^{-2}$ & $4$ & $2.15 \cdot 10^{-2}$ & $2$ & $4.93 \cdot 10^{-2}$ & $4$\\
4,865 & $7.12 \cdot 10^{-3}$ & $4$ & $9.59 \cdot 10^{-3}$ & $4$ & $2.28 \cdot 10^{-2}$ & $4$\\
19,201 & $4.18 \cdot 10^{-3}$ & $8$ & $5.63 \cdot 10^{-3}$ & $4$ &$1.28 \cdot 10^{-2}$ & $8$\\
76,289 & $1.78 \cdot 10^{-3}$ & $8$ & $2.44 \cdot 10^{-3}$ & $8$ &$5.82 \cdot 10^{-3}$ & $8$\\
304,129 & $1.06 \cdot 10^{-3}$ & $16$ & $1.51 \cdot 10^{-3}$ & $8$ & $3.36 \cdot 10^{-3}$ & $8$\\
1,214,465 & $4.82 \cdot 10^{-4}$ & $16$ & $6.59 \cdot 10^{-4}$ & $16$ & $1.59 \cdot 10^{-3}$ & $16$
\end{tabular}

\vspace*{3mm}

\begin{tabular}{r|rr|rr|rr}
\multicolumn{7}{c}{non-smooth problem}\\[1mm]
DoFs & $L^2$-error & $m$ & $L^\infty$-error & $m$ & $H^1$-error & $m$ \\ \hline
91 & $4.50 \cdot 10^{-2}$ & $4$ & $1.03 \cdot 10^{-1}$ & $4$ & $2.02 \cdot 10^{-1}$ & $2$\\
329 & $1.62 \cdot 10^{-2}$ & $4$ & $5.69 \cdot 10^{-2}$ & $4$ & $1.51 \cdot 10^{-1}$ & $4$\\
1,249 & $7.11 \cdot 10^{-3}$ & $8$ & $3.08 \cdot 10^{-2}$ & $8$ & $1.21 \cdot 10^{-1}$ & $8$\\
4,865 & $3.35 \cdot 10^{-3}$ & $16$ & $2.03 \cdot 10^{-2}$ & $16$ & $9.80 \cdot 10^{-2}$ & $16$\\
19,201 & $1.70 \cdot 10^{-3}$ & $32$ & $1.38 \cdot 10^{-2}$ & $32$ &$7.91 \cdot 10^{-2}$ & $32$\\
76,289 & $9.63 \cdot 10^{-4}$ & $32$ & $9.12 \cdot 10^{-3}$ & $32$ &$6.33 \cdot 10^{-2}$ & $32$\\
304,129 & $5.10 \cdot 10^{-4}$ & $64$ & $6.04 \cdot 10^{-3}$ & $64$ & $5.12 \cdot 10^{-2}$ & $64$\\
1,214,465 & $3.12 \cdot 10^{-4}$ & $64$ & $4.51 \cdot 10^{-3}$ & $64$ & $4.25 \cdot 10^{-2}$ & $64$
\end{tabular}
\end{center}
\caption{The second column shows the smallest relative $L^2$ error for a given grid across the \blue{factors} $m \in \{ 2, 4, 8, 16, 32, 64\}$, with the minimizing $m$ listed in the third column. The remaining columns are structured analogously.}
\label{tab:error}
\end{figure}

\begin{figure}
\begin{center}
\begin{tikzpicture}
    \begin{axis}[
        xlabel={Quartic problem: Number of refinements},
        ylabel={$\| u - u_h \|_{\infty} / \| u \|_{\infty}$},
        ymode=log,
        legend entries={$m=2$,$m=4$,$m=8$,$m=16$,$m=32$,$m=64$},
        legend pos=south west,
        x=1.2cm,y=0.75cm,
        ymin=0.0005, ymax=0.3,
        cycle list name=black white,
    ]

    \addplot plot coordinates {
        (0, 0.09264686) (1, 0.03848972) (2, 0.02147328) (3, 0.01595599) (4, 0.01387613) (5, 0.01299948) (6, 0.01260062) (7, 0.01241146) }; 
     \addplot plot coordinates {
        (0, 0.09476099) (1, 0.0730312 ) (2, 0.02412993) (3, 0.00959138) (4, 0.00563159) (5, 0.00452801) (6, 0.00418338) (7, 0.00405839) }; 
     \addplot plot coordinates {
        (0, 0.11693424) (1, 0.08072497) (2, 0.06444549) (3, 0.01997167) (4, 0.00606306) (5, 0.00243683) (6, 0.00150851) (7, 0.0012695 ) }; 
     \addplot plot coordinates {
        (0, 0.17739561) (1, 0.10839021) (2, 0.07419719) (3, 0.06175568) (4, 0.0186791 ) (5, 0.00506872) (6, 0.00154685) (7, 0.00065936) }; 
     \addplot plot coordinates {
        (0, 0.17727959) (1, 0.16839183) (2, 0.10242908) (3, 0.07215608) (4, 0.06045472) (5, 0.01824949) (6, 0.00477594) (7, 0.00128875) }; 
     \addplot plot coordinates {
        (0, 0.17727959) (1, 0.16825611) (2, 0.16441069) (3, 0.10050325) (4, 0.07109564) (5, 0.05979031) (6, 0.01808672) (7, 0.00468619) }; 
    \end{axis}
\end{tikzpicture}

\bigskip

\begin{tikzpicture}
    \begin{axis}[
        xlabel={Non-smooth problem: Number of refinements},
        ylabel={$\| u - u_h \|_{\infty} / \| u \|_{\infty}$},
        ymode=log,
        legend entries={$m=2$,$m=4$,$m=8$,$m=16$,$m=32$,$m=64$},
        legend pos=south west,
        x=1.2cm,y=0.75cm,
        ymin=0.0005, ymax=0.3,
        cycle list name=black white,
    ]

    \addplot plot coordinates {
        (0, 0.10296365) (1, 0.07818203) (2, 0.08125336) (3, 0.0816648 ) (4, 0.07860496) (5, 0.07660087) (6, 0.07549515) (7, 0.07487687) }; 
    \addplot plot coordinates {
        (0, 0.10296365) (1, 0.05692912) (2, 0.0497189 ) (3, 0.04769402) (4, 0.04583895) (5, 0.04460013) (6, 0.04364363) (7, 0.04304393) }; 
    \addplot plot coordinates { 
        (0, 0.11570551) (1, 0.05955971) (2, 0.03078344) (3, 0.02865993) (4, 0.02741087) (5, 0.02637367) (6, 0.02567876) (7, 0.02521068) }; 
    \addplot plot coordinates {
        (0, 0.233414  ) (1, 0.10945945) (2, 0.03919945) (3, 0.02029866) (4, 0.01570089) (5, 0.01509889) (6, 0.0140704 ) (7, 0.01356076) }; 
    \addplot plot coordinates {
        (0, 0.23590807) (1, 0.23491063) (2, 0.10710124) (3, 0.03358211) (4, 0.01377854) (5, 0.00911636) (6, 0.00753515) (7, 0.00693339) }; 
    \addplot plot coordinates {
        (0, 0.23590807) (1, 0.24051648) (2, 0.24970561) (3, 0.09912176) (4, 0.03142673) (5, 0.01151443) (6, 0.00603893) (7, 0.00450505) }; 
    \end{axis}
\end{tikzpicture}

\end{center}
\caption{Relative $L^\infty$-error for the test problem with quartic (above) and non-smooth (below) exact solution.}
\label{fig:rates}
\end{figure}
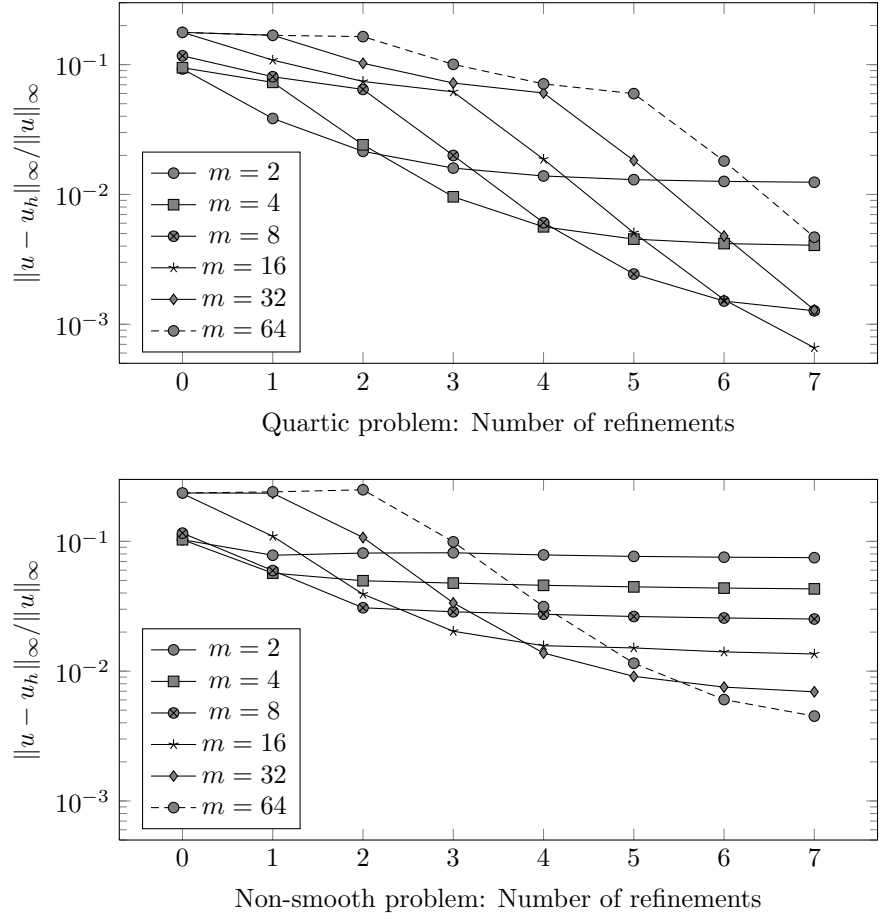
The relative errors in the $L^2$, $L^\infty$ and $H^1$ norms when approximating the quartic and non-smooth exact solution are summarized in the \cref{tab:error}.  The $L^\infty$-error graphs for different $m$ are plotted in \cref{fig:rates}. Across the seven levels of refinement the orders of convergence in $h$ and $k$ are, with $C$ representing generic constants:\\[-1mm]

\begin{center}
\begin{tabular}{l|ll|ll}
& \multicolumn{2}{c|}{quartic problem} & \multicolumn{2}{c}{non-smooth problem}\\ \hline
$\| u - u_h \|_{L^2}$ \rule{0pt}{2.3ex} & $\approx C h^{1.02}$ & $\approx C k^{1.79}$ & $\approx C h^{1.02}$ & $\approx C k^{2.39}$\\[0.5mm]
$\| u - u_h \|_{L^\infty}$ & $\approx C h^{1.02}$ & $\approx C k^{1.78}$ & $\approx C h^{0.64}$ & $\approx C k^{1.50}$\\[0.5mm]
$\| u - u_h \|_{H^1}$ & $\approx C h^{0.98}$ & $\approx C k^{1.72}$ & $\approx C h^{0.33}$ & $\approx C k^{1.22}$
\end{tabular}
\end{center}

\vspace*{2mm}

The number of Newton iterations in \cref{newton_iter} increases only moderately with the level of refinement and stencil size, so that fine meshes remain feasible on desktop computers. Importantly, Howard's algorithm displays a robust performance when approximating the non-smooth solution $|x_1|$ with $f = 0$; noting that the line $\{ x_1 = 0 \}$ where $|x_1|$ is non-differentiable is not aligned with the computational mesh. The iterations are started with the control $\mathbf{B}_0 = \frac{1}{d} \Id$. Due to global convergence, the starting iterate does not need to be guessed in close vicinity of the numerical solution. The stopping criterion is an iteration step size less than $10^{-6}$ in the $L^\infty$-norm.

\vspace*{1mm}
\begin{figure}
\begin{center}
\begin{tabular}{c|cccccc|cccccc}
& \multicolumn{6}{c|}{$m$ for quartic problem} & \multicolumn{6}{c}{$m$ for non-smooth problem}\\
refinement & $2 $ & $4$ & $8$ & $16$ & $32$ & $64$ & $2 $ & $4$ & $8$ & $16$ & $32$ & $64$\\ \hline
0 & \fbox{5} & 5 & 5 &  4 &  5 &  5 & 4 & \fbox{5} & 6 & 5 & 5 &  5\\
1 & \fbox{5} & 5 & 6 & 10 &  5 &  5 & 4 & \fbox{5} & 6 & 7 & 9 &  6\\
2 & \fbox{5} & 5 & 7 &  9 & 12 &  5 & 5 & 5 & \fbox{6} & 6 & 7 & 11\\
3 & 5 & \fbox{6} & 7 &  9 & 12 & 13 & 5 & 5 & 7 & \fbox{7} & 8 &  9\\
4 & 5 & \fbox{6} & 7 & 11 & 12 & 16 & 7 & 5 & 6 & 7 & \fbox{7} &  7\\
5 & 6 & 6 & \fbox{6} & 10 & 12 & 15 & 8 & 6 & 6 & 7 & \fbox{8} &  8\\
6 & 5 & 6 & \fbox{6} &  9 & 12 & 15 & 7 & 6 & 6 & 7 & 8 &  \fbox{8}\\
7 & 5 & 5 & 7 &  \fbox{8} & 10 & 14 & 8 & 5 & 7 & 7 & 8 &  \fbox{9}
\end{tabular}
\end{center}
\caption{Number of Newton iterations to achieve a Newton step size of less than $10^{-6}$. The boxes highlight the \blue{factor} $m$ which minimizes the $L^\infty$-error for a given level of refinement.}
\label{newton_iter}
\end{figure}

\bibliographystyle{siamplain}

\end{document}